\documentclass[12pt,reqno]{amsart}
\usepackage{amsmath}
\usepackage{amsfonts}
\usepackage{a4wide}
\usepackage{amsmath,amsthm,amssymb,amscd}
\usepackage{latexsym}
\usepackage{hyperref}
\usepackage[numbers,sort&compress]{natbib}
\usepackage{hypernat}
\allowdisplaybreaks
\numberwithin{equation}{section}
\usepackage[dvipsnames]{xcolor}

\usepackage{changes}

\newtheorem{theorem}{Theorem}[section]
\newtheorem{proposition}[theorem]{Proposition}
\newtheorem{corollary}[theorem]{Corollary}
\newtheorem{lemma}[theorem]{Lemma}

\theoremstyle{definition}

\newtheorem{definition}[theorem]{Definition}

\newtheorem{remark}[theorem]{Remark}

\newcommand{\ds}{\displaystyle}

\newcommand{\eq}{\eqref}

\def\R{\mathbb R}
\def\N{\mathbb N}
\def\C{\mathbb C}
\def\Z{\mathbb Z}

\def\a{\alpha}
\def\be{\beta}
\def\ga{\gamma}
\def\de{\delta}

\def\la{\lambda}
\def\si{\sigma}

\def\diverg{{\rm div}\,}

\newcommand{\dist}{{\rm dist}}
\newcommand{\va}{\varepsilon}

\usepackage{color}

\newtheorem{example}{\textbf{Example}}

\definechangesauthor[color=blue]{R.M.}

\usepackage{todonotes}

\newcommand{\beq}{\begin{equation}}
\newcommand{\eeq}{\end{equation}}
\newtheorem{notations}[theorem]{Notations}

\usepackage[dvipsnames]{xcolor}

\newcommand\myshade{85}
\colorlet{mylinkcolor}{violet}
\colorlet{mycitecolor}{YellowOrange}
\colorlet{myurlcolor}{Aquamarine}

\hypersetup{
	linkcolor  = mylinkcolor!\myshade!black,
	citecolor  = mycitecolor!\myshade!black,
	urlcolor   = myurlcolor!\myshade!black,
	colorlinks = true,
           }
   \def\bl{\color{blue}}

   \def\cL{{\mathcal L}}

   \def\sgn{\mathop{\rm sgn}\nolimits}
   \def\loc{\mathop{\rm loc}\nolimits}
   \def\spant{\mathop{\rm span}\nolimits}

\begin{document}

%%%%%%%%%%%%%%%%%%%%%%%%%%%%%%%%%%%%%%%%%%%%%%%%%%%%%%%%%%%%%%%%%%%%%%%%%%%%%%%%%%
%%%%%%%%%%%%%%%%%%%%%%%%%%%%%%%%%%%%%%%%%%%%%%%%%%%%%%%%%%%%%%%%%%%%%%%%%%%%%%%%%%

\title[BEC-Magnetic]
{Qualitative analysis of energy ground states for magnetic focusing Gross-Pitaevskii equations} 

\author{Xiaoming An}
\address{
\newline\indent Xiaoming An, School of Mathematics and Statistics, Guizhou University of Finance and Economics,
\newline\indent Guiyang, 550025, P. R. China.}
\email{xman@mail.gufe.edu.cn}

\author{Lun Guo}
\address{
\newline\indent Lun Guo, School of Mathematics and Statistics, South-Central Minzu University, 
\newline\indent Wuhan, 430074, P. R. China.}
\email{lguo@mails.ccnu.edu.cn}

\author{Xiao Luo}
\address{
\newline\indent Xiao Luo, School of Mathematics, Hefei University of Technology,
\newline\indent Hefei, 230009, P. R. China.}
\email{luoxiao@hfut.edu.cn}

\author{Riccardo Molle$^*$ }
\address{
\newline\indent Riccardo Molle, Dipartimento di Matematica, Universit\`{a} di Roma ``Tor Vergata'',
\newline\indent Via della Ricerca Scientifica n. 1, 00133,  Roma, Italy.}
\email{molle@mat.uniroma2.it}

\thanks{$^*$ Corresponding author}
\date{November 12, 2025}

\begin{abstract}
{We prove the uniqueness, asymptotics, symmetry, and orbital stability of energy ground states for 3D magnetic Gross-Pitaevskii equations under mild conditions on the electric potential and magnetic field. 
In particular, both the electric potential and the magnetic field are allowed to have singularities, and we cover in a unified approach the physically relevant Aharonov-Bohm magnetic field as well as the constant magnetic field. 
Since we are in the 3D case, the unique energy ground state (up to a phase factor) is obtained as a local minimizer, rather than a global one, by restricting the kinetic energy of candidate critical points within a suitable range. 
The qualitative analysis of the energy ground state we carry on is mainly based on a related Pohozaev identity, the implicit function theorem, and variational methods.}
\medskip

\medskip

\end{abstract}

\maketitle

%%%%%%%%%%%%%%%%%%%%%%%%%%%%%%%%%%%%%%%%%%%%%%%%%%%%%%%%%%%%%%%%%%%%%%%%%%%%%%%%%%

%\hrulefill 

%{\bl If you activate the following command (in the source file), you get the list of my changes.}
%\listofchanges

%\hrulefill 

%%%%%%%%%%%%%%%%%%%%%%%%%%%%%%%%%%%%%%%%%%%%%%%%%%%%%%%%%%%%%%%%%%%%%%%%%

%%%%%%%%%%%%%%%%%%%%%%%%%%%%%%%%%%%%%%%%%%%%%%%%%%%%%%%%%%%%%%%%%%%%%%%%%%%%%%%%%%

\tableofcontents

%%%%%%%%%%%%%%%%%%%%%%%%%%%%%%%%%%%%%%%%%%%%%%%%%%%%%%%%%%%%%%%%%%%%%%%%%%%%%%%%%%

\section{Introduction and main results}
%\label{s1}
In this paper, we study the existence and qualitative properties of energy ground states for the following magnetic Gross-Pitaevskii equation
\begin{equation}\label{eq1.1}
\left\{
  \begin{array}{ll}
    (i\nabla   + A(x))^2 u + (V(x)-\la)u = |u|^2u, &\ \text{in}\ \R^3  \\
    \ds\int_{\R^3}|u|^2=c, &
  \end{array}
\right.
\end{equation}
where $i^2=-1$, $V:\R^3\to\R$ is an electric potential, $A:\R^3\to\R^3$ is a magnetic potential associated to the magnetic filed $B$ by the relation $B=\nabla\times A$, $\la$ is an unknown Lagrange multiplier, arising from the  mass constraint $\|u\|_2^2=c>0$, and the magnetic laplacian is defined by
\beq
\label{1205}
(i\nabla + A(x))^2 = -\Delta + 2iA(x)\cdot\nabla + i\, \diverg A(x)+|A(x)|^2.
\eeq

\medskip

Bose-Einstein condensation (BEC) is a phenomenon occurring at very low temperature for a highly dilute gas of bosons. 
Under the appropriate experimental conditions, most of the particles get to occupy the same quantum state.
The problem \eq{eq1.1} arises as an effective model for large dilute bosonic systems, as a consequence of the occurrence of BEC in the energy ground states. 
In BEC, the magnetic structure is involved in scattering, superfluid, quantized vortices in plasma physics \cite{PRL}. 
For the early physical background of \eq{eq1.1}, see \cite{AHS} and the references therein.

In particular, the uniqueness of energy ground states ensures a complete BEC on it. 
On the other hand, in derivation of Hartree's theory for mean-field Bose systems, to prove that many-body ground states converge (in terms of reduced density matrices) to those of the NLS functional globally, rather than up to a subsequence, the uniqueness of energy ground states usually appears as an assumed condition; see \cite{LNR1,LNR2,LNR3,LNR4}. 
Note that if $A = 0$ and $V$ is radial, one can prove the uniqueness of the ground state of the NLS by well-known ODE arguments, reviewed, for instance, in \cite{Frank}. 
Uniqueness can certainly fail when $A\neq 0$ (due to the occurrence of quantized vortices \cite{Sei}), or when $V$ has several isolated minima \cite{GS}.

\smallskip

We would also like to mention that, in the repulsive case, the Gross-Pitaevskii equation, coupled with a Poisson equation, was proposed in cosmology and astrophysics to describe the dynamics of cold dark matter made of axions or bosons in the form of Bose-Einstein condensates at absolute zero temperatures.
We can refer to the surveys \cite{Chavanis15,POM20} for the physical model and to \cite{MMRarXiv,LM22}, and references therein, for asymptotic analysis.

\medskip

In this paper, we prove the existence, uniqueness, asymptotics, symmetry, and orbital stability of the energy ground states for \eq{eq1.1}, under mild conditions on the electric potential and magnetic field.
Following \cite{Dovetta-Serra-Tilli-MA-2023},  we say that $u\in{S_c}$ is an energy ground state of problem \eq{eq1.1} if
\begin{equation*}%\label{eeq1.7}
(\mathcal{E}|_{{S_c}})'(u)=0\ \text{and}\ \mathcal{E}(u)=\inf\{\mathcal{E}(v)\, :\, v\in{S_c}, \,(\mathcal{E}|_{{S_c}})'(v)=0\},
\end{equation*}
where
\[
\mathcal{E}(v)=\frac{1}{2}\int_{\R^3}|(i\nabla + A(x))v|^2 + V(x)|v|^2 - \frac{1}{4}\int_{\R^3}|v|^4 \ \text{ and }\  {S_c}=\{v\in \Sigma_{A,V}\, :\, \|v\|^2_2=c\}
\]
(see \ref{SNot} Notations).

\medskip

For the sake of completeness, we shall first consider the existence. 
In the 2D case, the authors in \cite{CH,LNR3} describe the precise many-body limiting behavior for the energy ground state of \eq{eq1.1} by considering the following minimizing problem
$$
\inf_{u\in S_c}\mathcal{E}(u).
\eqno{(\mathcal{N}_c)}
$$
However, turning to the 3D case, it holds $\inf\limits_{{u\in S_c}}\mathcal{E}(u) = -\infty$ in physically most relevant cases such as $V(x) = |x|^2 - \rho^2(x^2_1+x^2_2)$ and $A(x)=\rho(-x_2,x_1,0)$ with $|\rho|<1$, as it is readily seen testing the functional on $u_\tau := \tau^{3/2}u(\tau\cdot)$, $\tau>0$, for a fixed ${u\in S_c}$. 
Hence, the global minimization problem $(\mathcal{N}_{{c}})$ no longer works.
We point out that the unboudedness from below of the energy functional occurs because the nonlinearity of the third  order is $L^2$-supercritical in dimension 3 (see \cite{Bellazzni-et.al.-CMP-2017,LM_PRSE_23,MRV22JDE,BMRV21,PPVV_21_JDE,Je97NA}, and references therein, for the non-magnetic case).

\medskip

To address this problem, motivated by \cite{Bellazzni-et.al.-CMP-2017,Bellazzini-Jeanjean-SJMA-2016}, we assume that $A$ and $V$ satisfy the following condition:
\begin{align}
\left\{
  \begin{array}{l}
A\in L^2_{\loc}(\R^3,\R^3)\cap C^1(\R^3\backslash\{0\}) \\[8pt]
V\in L^1_{\loc}(\R^3)\cap C^1(\R^3\backslash\{0\})\ \mbox{ with }\ \lim\limits_{|x|\to\infty}V(x) = \infty\\
\exists \bar\alpha>0 \ : \  V(x)\ge \bar\alpha|A(x)|^2\ \ \forall x\in\R^3\setminus\{0\};
\end{array}
\right.
\tag{$\mathcal{V}_1$}\label{V1}
\end{align}

\begin{align}
\left\{
  \begin{array}{cl}
  (a) &  \sup\limits_{x\in{\R^3\backslash\{0\}}}\frac{|x\cdot \nabla V(x)| + \sum\limits_{j=1}^3|x|^2|\nabla A_j|^2}{V(x)}<+\infty, \\[4pt]
    (b) &\exists\, \va_0\in\big(0,\frac{1}{3}\big),\, \alpha_0>0 \ :\  \inf\limits_{x\in{\R^3\backslash\{0\}}}\frac{\tilde{T}_{A,V,\va_0}(x)}{V(x)}\ge \alpha_0, 
  \end{array}
\right.
\tag{$\mathcal{V}_2$}\label{V2}
\end{align}
where, for $x\neq 0$,
\beq
\label{Po}
\begin{split}
\tilde{T}_{A,V,\va_0}(x) :=& \Big[\frac{1}{2}V(x) + \frac{1}{6}x\nabla V(x)\Big] \\
& +  \Big[\big(\frac{1}{2}- \frac{1}{3\va_0}\big)|A(x)|^2 - \frac{1}{6\va_0}\sum_{j=1}^3|x\cdot \nabla A_j|^2   + \frac{1}{3}\sum_{j=1}^3A_j(x)\big(x\cdot \nabla A_j(x)\big)\Big].
\end{split}
\eeq
We point out that $\tilde{T}_{A,V,\va_0}$ originates from the Pohozaev identity related to \eq{eq1.1} (see \eq{Aeq3.25}).

\smallskip

We shall find the ground state solution as a local minimizer. 
More precisely, we first study the local minimization problem
\begin{equation}\label{abeq1.2}
m^{r}_{c} :=\inf_{u\in {S_c}\cap B(r)}\mathcal{E}(u),\qquad r>0,
\end{equation}
where
$$
B(r)=\left\{u\in\Sigma_{A,V}:\|u\|^2_{\dot{\Sigma}_{A,V}}:=\int_{\R^3}|(i\nabla + A(x))u|^2 + V(x)|u|^2\le r\right\}.
$$
Then, when 
\begin{equation*}
\mathcal{M}^r_{{c}}:=\{u\in {S_c}\cap B(r):\mathcal{E}(u)=m^r_{{c}}\}\ne\emptyset,
\end{equation*}
we prove that the local minimizers are not on the boundary of ${S_c}\cap B(r)$, so that  $m^r_{{c}}$ is indeed a critical value of $\mathcal{E}|_{{S_c}}$.
Finally, in order to obtain a ground state solution in $\mathcal{M}^r_{{c}}$, it remains to prove that
$$
m^r_c=\inf\{\mathcal{E}(v):v\in{S_c}, (\mathcal{E}|_{{S_c}})'(v)=0\}.
$$
To state the stability of $u_{c,A,V}$, let us fix some notations. 
\begin{definition} (see \cite{Esteban-Lions-1989,Cazenave-Esteban-MAP-1988,Cazenave-AMS-2003}) 
A set  $Y\subset{{\Sigma}}_{A,V}$ is called orbitally stable under the flow associated with
\begin{equation}\label{keq1.1}
\left\{
  \begin{array}{ll}
    i\psi_t = (i\nabla + A(x))^2\psi + V(x)\psi - |\psi|^2\psi, & (t,x)\in\R^+\times \R^3, \\
    \psi(0,x) = u_0(x), &
  \end{array}
\right.
\end{equation}
if for any $\va>0$, there exists $\delta>0$ such that for any initial data $u_0\in{{\Sigma}}_{A,V}$ satisfying
$$
\inf_{\phi\in Y}\|u_0-\phi\|_{{\Sigma}_{A,V}}<\delta,
$$
the corresponding solution $u(t,\cdot)$ of problem \eq{keq1.1} exists globally in time and satisfies
$$
\sup_{t\ge 0}\text{dist}_{{\Sigma}_{A,V}}(u(t,\cdot),Y)<\va.
$$
\end{definition}

\begin{definition} 
%\label{D1.2}
    Problem \eqref{keq1.1} is said to be locally well-posed in $\Sigma_{A,V}$ if, for any $u_0\in\Sigma_{A,V}$, there exists $T\in(0,\infty]$ and a unique solution $\psi(t,x)\in C([0,T),\Sigma_{A,V})$ of \eq{keq1.1} with $\psi(0,\cdot) = u_0$. 
%\deleted[id=R.M.]{Moreover, we say that 
%{\ye \eqref{keq1.1} satisfies the blow-up alternative property if %either $T=\infty$ or $T<\infty$ and $\lim\limits_{t\to %T^{-}}\|\psi(t,\cdot)\|_{{\dot{\Sigma}}_{A,V}} = \infty$.}}
\end{definition}

\medskip

Let us state the existence and stability result.

\begin{theorem}\label{th1.1}
For any fixed $r>0$:

\smallskip

\begin{itemize}
    
\item[1.] let $A,V$ satisfy \eq{V1}, then there exists a explicit $c_{A,V}(r)>0$ such that \eq{eq1.1} has a couple $(u_{{c},A,V},\la_{{c},A,V})$ of weak solution with $u_{{c},A,V}\in\mathcal{M}^{r}_{{c}}$ if $0<{c}<{c}_{A,V}(r)$;

\smallskip

\item[2.] let $A,V\in C^{1,\a}(\R^3\backslash\{0\})$ for some $\a\in (0,1)$ and satisfy \eq{V1}-\eq{V2}, then there exists a explicit $\tilde{c}_{A,V}(r)\leq c_{A,V}(r)$ such that $u_{{c},A,V}$ is an energy ground state of \eq{eq1.1} if $0<c<\tilde{c}_{A,V}(r)$;

\smallskip

\item[3.] assume, moreover, that \eq{keq1.1} is locally well-posed in $\Sigma_{A,V}$ with solutions preserving both mass and energy,  then the set $\mathcal{M}^r_{{c}}$ is orbitally stable under the flow corresponding to \eq{keq1.1} if $0<c<c_{A,V}(r)$.
\end{itemize}
\end{theorem}

    Antonelli et al. in  \cite{Antonelli-Marahrens-DCDS-2012,Antonelli-Carles-Silva-CMP-2015} proved that \eq{keq1.1}  is locally well-posed in $\Sigma_{A,V}$, for the physically relevant case presented in the following Example \ref{ex1}. 
    Moreover, they showed that if $u(t,x)$ is a solution, then the mass and energy are preserved for all $t\in[0,T)$, where $T=\infty$ or $T<\infty$ and $\lim\limits_{t\to T^{-}}\|u(t,\cdot)\|_{{\dot{\Sigma}}_{A,V}} = \infty$.

\medskip

 Appendix B provides an estimate of $c_{A,V}(r)$ and $\tilde c_{A,V}(r)$. 
In particular,  see \eq{1536} and \eq{1537},
\beq
\label{1354}
\sup_{r>0}c_{A,V}(r)=:c_*\in (0,\infty),\qquad  \sup_{r>0}\tilde c_{A,V}(r)=:\tilde c_*\in (0,c_*].
\eeq
Then, the following straightforward consequence of Theorem \ref{th1.1} can be stated.

\begin{theorem}\label{th1.1Bis}
\begin{itemize}
\item[1.] If $A,V$ satisfy \eq{V1}, then there exists a explicit $c_*>0$ such that \eq{eq1.1} has a couple $(u_{{c},A,V},\la_{{c},A,V})$ of local minimum weak solution for all  $0<{c}<{c}_*$;

\smallskip

\item[2.] let $A,V\in C^{1,\a}(\R^3\backslash\{0\})$ for some $\a\in (0,1)$ and satisfy \eq{V1}-\eq{V2}, then there exists a explicit $\tilde{c}_*\leq c_*$ such that $u_{{c},A,V}$ is an energy ground state of \eq{eq1.1} if $0<c<\tilde{c}_*$.
\end{itemize}
\end{theorem}

Let us observe that, in general, we cannot expect that for every $c>0$ there exists $r(c)>0$ such that the infimum $m_{r(c)}^c$ is achieved, in the interior of  ${S_c}\cap B(r)$, as the following result shows.

\begin{proposition}
\label{cr1.4} 
Let $A,V$ satisfy \eq{V1}-\eq{V2}, and let us denote
$$
c^* = 2\frac{\hat{\la}_{1,A,V}}{\|\hat{\psi}_{1,A,V}\|^4_4},
$$
where $\hat{\psi}_{1,A,V}$ is the first eigenfunction of $\hat{L}_{A,V} = (i\nabla + A)^2 + V(x)$ (see Lemma \ref{Ale2.8}).
If $c\ge c^*$, then the infimum $m^r_c$ cannot be achieved by a solution of \eq{eq1.1}, for any $r>0$.
%Especially,  all possible energy ground states of \eq{eq1.1} are not local minimizers,  if $c\ge c^*$.
\end{proposition}
The proof of Proposition \ref{cr1.4} is contained in \S \ref{S3}.

\begin{remark}
    \label{Rreg}
The regularity issue of a solution $u\in\Sigma_{A,V}$ of \eqref{eq1.1} is quite standard. 
We refer to \cite[Propositions 2.2. and 2.5]{CS05TMNA} for $u\in L^\infty(\R^3)$ and the exponential decay of $|u|$ (see also \cite{K97MathZ}).
In particular, taking also into account \eqref{1205}, by Moser's iteration technique, it is standard to prove $u\in W^{2,q}_{\loc}(\R^3)$ for $2\le q <\infty$, and therefore to $C^{1,\alpha}(\R^3\backslash\{0\})$ for any $\alpha\in (0,1)$.
If a more regularity assumption is imposed on $A$ and $V$, i.e. $A\in C^{1,\alpha}(\R^3\backslash\{0\})$ and $V\in C^{0,\alpha}(\R^3\backslash\{0\})$, then $u$ is a classical solution in $\R^3\backslash\{0\}$.
\end{remark}

The next examples illustrate physical scenarios that satisfy \eq{V1} and \eq{V2},
that is, this work presents a unified approach to these physical problems.

\begin{example}\label{ex1}
 Let $x=(x_1,x_2,x_3)\in\R^3$, and denote $x^{\bot} =(-x_2,x_1,0)$. 
 A typical class of examples is, for $k\in\N^+$  and $\rho_k\in\R$,
\beq
\label{1111}
A_{\rho_k}(x) = \rho_kx^{\bot},\qquad x\in\R^3, 
\eeq
with potential of the form
$
V_{\rho_k}(x) = |x|^{2k} + k\sgn(k-1) - \rho^2_k |x^{\bot}|^2.
$

\smallskip

When $A$ is of type \eqref{1111}, equation \eq{eq1.1} is used in \cite{X.Luo-T.Yang-JDE-2020,Bellazzni-et.al.-CMP-2017,Cazenave-AMS-2003,Guo-Li-Luo-JFA-2024,Guo-Luo-Yang-ARMA,Guo-Luo-Peng-SIAM-2023} to describe rotating BECs in different traps.
Letting $|\rho_k|<1$, since $V_{\rho_k}(x) \ge (1-\rho^2_k)|A_{\rho_k}(x)|^2$, condition \eq{V1} is verified. 
Obviously,   \eq{V2}$(a)$ holds.
Furthermore, if $|\rho_k|<\min\big\{\frac{\sqrt{3+2k}}{3},1\big\}$ and $\va_0 \in \Big(\frac{3}{3+2k}\rho^2_k,\frac{1}{3}\Big)$, it holds that
$$
\inf_{|x|>0}\frac{\tilde{T}_{A_{\rho_k},V_{\rho_k},\va_0}}{V_{\rho_k}} = \inf_{|x|>0}\frac{\Big(\frac{1}{2}+\frac{k}{3}\Big)|x|^{2k} + \frac{k\sgn(k-1)}{2}- \frac{\rho^2_k}{2\va_0}|x^{\bot}|^2}{V_{\rho_k}(x)}
 >\min\left\{\frac{1}{2}+\frac{k}{3}-\frac{\rho_k^2}{2\va_0},\frac{1}{2}\right\},
$$
which gives \eq{V2}$(b)$, too.
\end{example}

\medskip

\begin{example}\label{ex2}\
Another physically relevant magnetic field is the following singular magnetic potential with homogeneity of order $-1$
$$
\mathcal{A}_{\hat{\a}}(x) = \frac{{\hat{\alpha}}}{|x|^2} \Big(-x_2,x_1,0\Big),\qquad x\in\R^3\setminus\{ 0\},\ \hat{\alpha}\in\R\backslash\{0\},
$$
which is  introduced in \cite{V.Felli-A.Ferrero-S.Terracini-JEMS-2011}. 
Let $V_\be(x) = \frac{\be}{|x|^2} + |x|^2$, with $\be>|\hat\alpha|$ be a parameter.  
Obviously, \eq{V1} and \eq{V2}$(a)$ hold. 
Moreover, a tedious computation shows that there is a positive constant $C_{\va_0}$ that depends only on $\va_0$  such that
$$
\inf_{|x|>0}\frac{\tilde{T}_{\mathcal{A}_{\hat{\a}},V_\be ,\va_0}(x)}{V_{\be}(x)}\ge \inf_{|x|>0}\frac{5|x|^2/6 + (\be/6- |{\hat{\a}}|^2C_{\va_0})|x|^{-2}}{|x|^2 + \be|x|^{-2}} = \frac{1}{6} - \frac{C_{\va_0}|{\hat{\a}}|^2}{\be}.
$$
Then \eq{V1} and \eq{V2} are satisfied by letting $\be > 6C_{\va_0}|\hat{\a}|^2$.
\end{example}

\medskip

%Theorem \ref{th1.1} implies that the limit of $u_{c,A,V}/\sqrt{c}$, as $c\to 0^+$, will be one of the first eigenfunction of $(i\nabla + A(x))^2 + V(x)$. Hence, it is reasonable of using implicit function theorem to derive
Next, we consider the uniqueness and symmetry of energy ground states for \eq{eq1.1}. 
Let us introduce $G = \{g_{\theta}:\theta\in[0,2\pi)\}$, where $g_\theta$ have the form
\begin{equation}\label{ueq1.13}
g_{\theta}^+ = \left(
               \begin{array}{ccc}
                 \cos\theta & -\sin\theta & 0\\
                 \sin\theta & \cos\theta  & 0 \\
                 0 & 0 &  1 \\
               \end{array}
             \right)
             \qquad \mbox{ or } \qquad
  g_{\theta}^- = \left(
               \begin{array}{ccc}
                 \cos\theta & -\sin\theta & 0\\
                 \sin\theta & \cos\theta  & 0 \\
                 0 & 0 &  -1 \\
               \end{array}
             \right).           
\end{equation}
Obviously, $G$ is a subgroup of $O(3)$. 

\smallskip

Subsequently, we say that a function $u:\R^3\to \R^3$ is free of vortices if it is divergence-free.
We have

\begin{theorem}\label{th1.3}
Let $A,V$ satisfy \eq{V1}-\eq{V2} and   $A,V\in C^{1,\a}(\R^3\backslash\{0\})$, for some $\a\in (0,1)$, and the first eigenvalue of $(i\nabla + A(x))^2 + V(x)$ is simple.

1. Then there exists ${c}_0>0$ such that the energy ground state of \eq{eq1.1}, up to a phase factor, is unique if $0<{c}<{c}_0$.

2. Assume, moreover, that $A$ is linear, skew-symmetric, $A\circ g = g\circ A$, $\forall g\in G$, and $V(x)+|A(x)|^2$ is cylindrically symmetric and radially nondecreasing w.r.t. $(x_1,x_2)$. 
Then the energy ground state of \eq{eq1.1} is cylindrically symmetric, positive up to a constant phase and free of vortices  if  $0<{c}<c_0$.

\end{theorem}

\begin{remark}
\label{R1.7}
For the two dimensional case, i.e., the problem
\begin{equation*}
%\label{eq1.1bis}
\left\{
  \begin{array}{ll}
    (i\nabla + A(x))^2 u + (V-\la)u = |u|^2u, &\ \text{in}\ \R^2  \\
    \ds\int_{\R^2}|u|^2=c, &
  \end{array}
\right.
\end{equation*}
we can also obtain results that are similar to the ones stated above, by the same proof as in this paper. 
In particular, this remark applies to Example 2, which in two dimensions provides the physically relevant Aharonov-Bohm (AB) magnetic field
$$
\hat{\mathcal{A}}_{\hat\a} = \hat\a\Big(-\frac{x_2}{|x|^2},\frac{x_1}{|x|^2}\Big),\quad x=(x_1,x_2)\in\R^2,\  \hat\a\in\R\setminus\{0\}.
$$
In \cite{Aharonov0Bohm-PRL-1959},  Aharonov and Bohm used the equation
\beq
\label{ABeq}
\left[\frac{\partial^2}{\partial r^2} + \frac{1}{r}\frac{\partial}{\partial r} + \frac{1}{r^2}\Big(\frac{\partial}{\partial \theta} + i\hat\a\Big)^2+k^2\right]\psi=0
\eeq
as a mathematical model to interpret experiments that show the importance of electromagnetic potentials in quantum theory. 
For a description of the physical model, we also refer to \cite{V.Felli-A.Ferrero-S.Terracini-JEMS-2011}. 
Equation \eq{ABeq} can be seen as the linear part of \eq{eq1.1} under (AB) potential into cylindrical coordinates when $k=0$ (see the decomposition in Lemma \ref{able5.1}, for example).
We need to point out that our method cannot tackle the Aharonov-Bohm magnetic potential in the three-dimensional case, because
$$
\hat{\mathcal{A}}_{\hat\a}(x) = \hat\a \Big(\frac{-x_2}{x^2_1+x^2_2},\frac{x_1}{x^2_1+x^2_2},0\Big),\quad x\in\R^3,\ x_1^2+x_2^2\neq 0,
$$
does not verify $|\hat{\mathcal{A}}_{\hat\a}|^2 \in L^1_{\loc}(\R^3)\cap C^1(\R^3\setminus\{0\})$. 
For further study of (AB) potential, we refer the reader to \cite{Bonheure-RMP}.
\end{remark}

\begin{remark}
%\label{re1.3}
Theorem \ref{th1.1} gives explicit ranges on the mass $c$ that ensure the existence of local minimizers and stable energy ground states, even if we are facing quite general electric potentials and magnetic fields; see detailed expression of the threshold on mass in Appendix \ref{AB}. 
This, on the one hand, generalizes the main results established in \cite{X.Luo-T.Yang-JDE-2020} dealing with nonlinear Schr\"{o}dinger equation with rotation. 
On the other hand,  it covers the case that $V$ and $A$ have singular points, especially the physically relevant AB potential (Example \ref{ex2} and Remark \ref{R1.7}), which, compared with known results, such as \cite{Guo-Zeng-Zhou-Poincare-2016,Guo-Luo-Peng-SIAM-2023,Guo-Luo-Yang-ARMA,Bellazzini-Jeanjean-SJMA-2016,Bonheure-Nys-Schaftingen-JMPA-2019,Bellazzni-et.al.-CMP-2017,X.Luo-T.Yang-JDE-2020},  is an interesting advance in the BEC theory.
\end{remark}

Extending the spectrum results for non-magnetic coupling Laplacian $-\Delta + V$ to the magnetic Laplacian case $(i\nabla + A)^2 + V$ has a long history, see, f.i., \cite{Shen-CMP-1996,B.Poggi-ADV-2024,Shen-DM-1998,Shen-IUMJ-1996} and the references therein. 
In Section \ref{S5} we presents sufficient conditions on $A$ to ensure the validity of the assumption that the first eigenvalue is simple.
In particular, we shall show that this property holds if $A = \mathcal{A}_{\hat{\a}}$ in Example \ref{ex2} with $|{\hat{\a}}|\le \frac{1}{2}$, see \cite{V.Felli-A.Ferrero-S.Terracini-JEMS-2011} for related research. 
Note that $\mathcal{A}_{\hat{\a}}$ is nonlinear and therefore does not satisfy the additional assumption of Part 2 of Theorem \ref{th1.3}. 
However, for the minimal solutions related to this family of magnetic potentials, we can obtain the same properties with the same discussion in Theorem \ref{th1.3}. Denote $V_{\hat{\a}} = V(x) + |\mathcal{A}_{\hat\a}(x)|^2$, then:

\begin{corollary}\label{abth1.5}
Let  $|{\hat{\a}}|\le 1/2$, and {$V_{\hat{\a}}$ satisfies \eq{V1}-\eq{V2} with $A(x) \equiv 0$}; moreover, assume that $V_{\hat{\a}}$ is cylindrically symmetric and nondecreasing w.r.t. $(x_1,x_2)$. 
If $0<c<c_{0,V_{\hat{\a}}}(r)$ then the infimum in \eq{abeq1.2} is achieved by a function $u_{c,0,V_{\hat\a}}$ that solves \eq{eq1.1}, is cylindrically symmetric w.r.t. $(x_1,x_2)$, and positive, up to a constant phase. Moreover, $u_{c,0,V_{\hat\a}}$ is unique up to a constant phase and free of vortices if $0<c<\hat{c}_0$ for some positive constant $\hat{c}_0$.
\end{corollary}

Actually, since $\mathcal{A}_{{\hat{\a}}}(x)\cdot \nabla u=0$ if $u$ is cylindrically symmetric w.r.t. $(x_1,x_2)$, we only need to consider the angular part and the cylindrical part separately in the term  $|(i\nabla + A_{\hat{\a}}(x))u|^2$  after rewriting $u$ in cylindrical coordinates (see \eq{eqA.9bis}). 
 
 Based on this observation and an argument similar to the one developed in Eq. (86) in \cite[Appendix C]{Triay-SJMA-2018}, we can convert \eq{eq1.1} into a problem without magnetic coupling effects:
\begin{equation}\label{abeq1.5}
\left\{
  \begin{array}{ll}
    -\Delta u + (V(x) + |\mathcal{A}_{\hat{\a}}|^2-\la)u = |u|^2u, &\ \text{in}\ \R^3  \\
    \ds\int_{\R^3}|u|^2=c. &
  \end{array}
\right.
\end{equation}
%(see also the remark after Theorem 7.4 in \cite{Lieb-Loss-AMS-2005}).

%\medskip

%\deleted[id=R.M.,comment={We have erased that Remark}]{More %symmetric properties of $u_{c,\mathcal{A}_{\hat{\a}},V}$ can be %expected if $V_{\hat{\a}}$ has a better structure, see Remark %\ref{abre4.2} for details.}

%\medskip

 For problem \eq{abeq1.5}, the special case $V(x)+|\mathcal{A}_{\hat{\a}}(x)|^2 = |x^{\bot}|^2=x^2_1+x^2_2$  has been discussed extensively, starting from \cite{Bellazzni-et.al.-CMP-2017} (see also \cite{MR4831725}, and references therein).
In particular, in \cite{Bellazzni-et.al.-CMP-2017}, the local minimizer for $m^r_c$ and the energy ground state for \eq{eq1.1} were obtained using the concentration-compactness argument, although $x^2_1+x^2_2$ in $\R^3$ does not meet the growth condition in \eq{V1}.
We point out that the method developed in this paper enables us to give an accurate range of $c$ in which a local minimizer exists, both in the case $V(x)+|\mathcal{A}_{\hat{\a}}(x)|^2 = |x|^2$ and in the case $V(x)+|\mathcal{A}_{\hat{\a}}(x)|^2 = |x^{\bot}|^2$ considered in \cite{Bellazzni-et.al.-CMP-2017}; see Remark \ref{reB.2} in Appendix \ref{AB} for details.
Problems of the type described above are also considered in dimension 2. 
We refer, for example, to \cite{Guo-Zeng-Zhou-Poincare-2016}, which addresses the problem \eq{abeq1.5} in $\R^2$ when $V(x)+|\mathcal{A}_{\hat{\a}}(x)|^2 = (|x|-A)^2$, $A>0$.

\bigskip

In \cite[Theorem 1.1]{X.Luo-T.Yang-JDE-2020}, some asymptotic behavior of the solution $(u_{c,A,V},\la_{c,A,V})$ as $c\to 0^+$ was exhibited in the cases $V(x) = |x|^2  - \rho^2|x^{\bot}|^2$ and $A(x) = \rho x^{\bot}$ (a special case of Example \ref{ex1}). 
These properties, according to our uniqueness result in Theorem \ref{th1.3} and the simplicity discussion in  \S\ref{S5}, can be improved as follows.

\begin{theorem}\label{th1.6}
Let $A_\rho(x)= \rho x^{\bot}$, $V(x) = |x|^2  - \rho^2|x^{\bot}|^2$,   and $(u_{c,A_{\rho},V}, {\la}_{c,A_{\rho},V})$ as in Theorem \ref{th1.1}. 
Then, there exist $\rho_0>0,c(\rho_0)>0$ such that 
$$
3- \la_{c,A,V} = O(c),\ {\Big\|u_{c,A,V} - l_{c,A,V}e^{-\frac{|x|^2}{2}}\Big\|^2_{{\Sigma}_{A,V}}} = o(c)
$$
if $|\rho|<\rho_0$ and $0<c<c(\rho_0)$, where
\[l_{c,A,V} = \frac{1}{\pi^{3/2}}\int_{\R^3}(u_{c,A,V})\,e^{-\frac{|x|^2}{2}}.\]%, $3$ is the first eigenvalue of simplicity of $-\Delta + |x|^2$ and $\frac{1}{{\pi^{3/4}}}e^{-\frac{|x|^2}{2}}$ is the corresponding first eigenfunction(\cite{Antonelli-Carles-Silva-CMP-2015}).
\end{theorem}

%\begin{corollary}
%%1. If $V(x) + |\mathcal{A}_\a(x)|^2 = |x^\bot|^2$, then \eq{eq1.1} has a couple of solution $(u_{c,0,|x^\bot|^2},\la_{c,0,|x^\bot|^2})$ to \eq{eq1.1} with $u_{c,0,|x^\bot|^2}\in \mathcal{M}^r_c$ if $0<c<c_{0,|x^\bot|^2}(r)$. There exists a $k\in \R$ such that $u_{c,0,|x^\bot|^2}(x_1,x_2,x_3-k)$ is positive, radially symmetric, radially decreasing w.r.t. $(x_1,x_2)$ and $x_3$. Moreover, $u_{c,0,|x^\bot|}$ is an energy ground state to \eq{eq1.1} if $0<c<\min\{r/6,c_{0,|x^\bot|^2}(r)\}$.
%
%2. If $V(x) + |\mathcal{A}_\a(x)|^2 = |x|^2$, then \eq{eq1.1} has a couple of solution $(u_{c,0,|x|^2},\la_{c,0,|x|^2})$ to \eq{eq1.1} with $u_{c,0,|x|^2}\in \mathcal{M}^r_c$ if $0<c<c_{0,|x|^2}(r)$. $u_{c,0,|x|^2}$ is positive, radially symmetric and decreasing. Moreover, $u_{c,0,|x|^2}$ is an energy ground state to \eq{eq1.1} if $0<\min\{r/9,c_{0,|x|^2}(r)\}$ and there exists a constant $\hat{c}_0>0$ such that $u_{c,0,|x|^2}(r)$ is unique up to a constant phase if $0<c<\hat{c}_0$.
%\end{corollary}
% Then the existence parts follow by the same argument of Theorem \ref{th1.1} and the work in \cite{Bellazzni-et.al.-CMP-2017}, see subsection \ref{ss4.2} for details. But, as stated in Remark \ref{re1.3}, we will give an explicit value for  $c_{0,|x^\bot|^2}(r)$ and $c_{0,|x|^2}(r)$ respectively and then an explicit range of the existence of energy ground states, which is an interesting improvement to the results in \cite{Bellazzni-et.al.-CMP-2017}, see Remark \ref{reB.2} in Appendix \ref{AB} for details.

One of the main contributions of this paper is addressing the uniqueness of energy ground states of \eq{eq1.1} under general conditions on $V$ and $A$. 
Based on the non-degeneracy and uniqueness of positive solutions of
\begin{equation}\label{uueq1.14}
-\Delta u + u = |u|^{p-2}u\ \ \text{in}\ \R^N
\end{equation}
(\cite{Kwong-1989} and \cite[\S 4.2]{AMRN}),
the authors in \cite{Bonheure-Nys-Schaftingen-JMPA-2019} obtained the uniqueness of ground states for
\begin{equation}\label{uueq1.13}
(i\nabla + A(x))^2 u + u = |u|^{p-2}u\ \text{in}\ \R^N
\end{equation}
when $2<p<2^*$, where $A$ is skew-symmetric and $|\nabla \times A|$ is small.
However, \eq{uueq1.13} is a $1$-frequency problem, which is different from \eq{eq1.1}. Actually, it seems not reasonable to have that a ground state  of \eq{uueq1.13} is an energy ground state of \eq{eq1.1} under general assumption on $V$. 
The equivalent relation between fixed frequency ground states and energy ground states  of \eq{uueq1.14} has been stated in \cite[Theorem 1.2]{Dovetta-Serra-Tilli-MA-2023}, that is the (negative) energy ground state level is the Legendre-Fenchel transform of the action ground state level.
See also \cite[Remark 2.18]{MMRarXiv} for the analogous issue in the defocusing Gross-Pitaevskii-Poisson equation.

\medskip

Recently, Guo et al. in \cite{Guo-Luo-Peng-arxiv-2023} used a finite-dimensional reduction method to obtain, for small $c$, a concentrating family of solutions to problem \eq{eq1.1}, when the electric potential is anharmonic and the magnetic field is constant. 
By constructing an auxiliary problem involving the Lagrange multipliers, that exhibit a blow-up phenomenon, they rescale in a suitable way and obtain the $1$-frequency problem \eq{uueq1.14} as a limiting problem, as $c\to 0^+$.
Then using the nondegeneracy of the positive solution of \eq{uueq1.14}, the uniqueness of the concentrating solution is obtained, when $c$ is small enough. 
We emphasize that Theorem \ref{th1.6} shows that the Lagrange multiplier blow-up phenomenon will not occur in our framework, even for the physically relevant case in Example \ref{ex1}.
Moreover, for the (AB)-type magnetic field and the singular electric potential presented in Example \ref{ex2}, it seems difficult to get a limiting problem with a good structure, as in \eq{uueq1.14}.
The approach described in \cite{Guo-Luo-Peng-arxiv-2023} is therefore not applicable in this context.

\medskip

% The main idea of the proof for the uniqueness part in Theorem \ref{th1.3} is the implicit function theorem.
% The construction of the auxiliary function is based on the decomposition of $\Sigma_{A,V}$, using the first eigenspace corresponding to the first eigenvalue of the linear part of our operator. 
% However, since we are in the complex value case, this approach is not straightforward. 
% With respect to the real case considered in \cite{Guo-Zeng-Zhou-Poincare-2016}, for example, in this context a non-trivial construction with respect to rotations is needed; see \eqref{9999} for details. 
% The construction here is based on a perturbation of mass, which is different from that of \cite{Triay-SJMA-2018} on the magnetic field. Indeed, the functional considering in \cite{Triay-SJMA-2018} is nonnegative and coercive by its admissible assumption on parameters. However, we are in a sign-indefinite case, this causes a lack of uniqueness for the related limiting problem without magnetic fields. Furthermore, symmetry of the energy ground state follows from the uniqueness result by carefully analysing the forms of $A$ and $V$ under suitable cylindrical coordinates.

The main idea of the proof for the uniqueness part in Theorem \ref{th1.3} is the implicit function theorem, recently used in \cite{Guo-Zeng-Zhou-Poincare-2016,Triay-SJMA-2018}. 
The construction of the auxiliary function is based on the decomposition of $\Sigma_{A,V} $ into suitable eigenspaces. 
However, since we are in the complex value case, the first eigenspace corresponding to the first eigenvalue is more complicated than that in \cite{Guo-Zeng-Zhou-Poincare-2016},   and a skillful construction with respect to rotations is needed, see \eqref{9999} for details. 
The construction here is based on a perturbation of mass, which is different from that of \cite{Triay-SJMA-2018} in the magnetic field. 
In fact, the functional considered in \cite{Triay-SJMA-2018} is nonnegative and coercive, by the admissible assumption on the parameters therein. 
However, we are in a sign-indefinite case, and this causes a lack of uniqueness for the related limiting problem without magnetic fields. 
Furthermore, symmetry of the energy ground state follows from the uniqueness result by a careful analysis of the forms of $A$ and $V$, in suitable cylindrical coordinates.

\medskip

% \vspace{0.3cm}

% \textbf{Plan of the paper.} In Section \ref{s2}, we give some preliminaries. In section \ref{s3}, we give the proof of Theorem \ref{th1.1}.
% In Section \ref{s4}, we prove Theorems \ref{th1.3} and \ref{abth1.5}. In Appendix A we will give some conditions on $A$ so that the assumptions in Theorem \ref{th1.3} are reasonable and then prove Theorem \ref{th1.6}. The spectrum analysis for $(i\nabla + A(x))^2 + V(x)$ is also settled in Appendix A. Appendix B presents explicit expression of the threshold on mass.

%%%%%%%%%%%%%%%%%%%%%%%%%%%%%%%%%%%%%%%%%%%%%%%%%%%%%
%%%%%%%%%%%%%%%%%%%%%%%%%%%%%%%%%%%%%%%%%%%%%%%%%%%%%
  
\section{Notations and preliminary results}
%\label{s2}

%%%%%%%%%%%%%%%%%%%%%%%%%%%%%%%%%%%%%%%%%%%%%%%%%%%%%
%%%%%%%%%%%%%%%%%%%%%%%%%%%%%%%%%%%%%%%%%%%%%%%%%%%%%

\begin{notations}\label{SNot}
$\phantom{pp}$

{\small
\begin{itemize}

\item[-)] $\C$ is the set of complex numbers. 
For each $v \in \C$, $\Re v$ refers to the real component of $v$, and $\bar{v}$ is its conjugate.

\smallskip

\item[-)] $\sgn v=\frac{\bar v}{|v|}$ for $v\in \C\setminus\{0\}$, $\sgn 0=0$; in particular,  $\sgn x=\frac{x}{|x|}$ for $x\in \R\setminus\{0\}$.

\smallskip

\item[-)] $H^1(\R^3,\C)$ and $L^p(\R^3,\C)$,  with $1\le p\le \infty$, are the standard Sobolev and Lebesgue spaces, with the norm $\|u\|$ and $\|u\|_p$, respectively.
We briefly write $H^1$ and $L^p$.

\smallskip

\item[-)] $H^1_{A} :=\left\{u\in L^2(\R^3,\C)\ :\ \|u\|_{A}<\infty\right\}$, where
$$
\|u\|^2_{A}:= \|(i\nabla + A(x))u\|^2_2 + \|u\|^2_2
$$
(see, f.i. \cite[\S 7.20]{Lieb-Loss-AMS} or \cite[\S 9]{Cazenave-AMS-2003}).
$H^1_{A} $ is a Hilbert space with inner product 
$$
\langle u,v\rangle_{A} = \Re\int_{\R^3}\big((i\nabla + A(x))u\cdot \overline{(i\nabla + A(x))v }+ u\bar{v}\big).
$$

\smallskip

\item[-)] $ \Sigma_{A,V} :=\Big\{u\in H^1_{A} \ :\ \int_{\R^3}V(x)|u|^2 <\infty \Big\},$ 
with inner product and norm
$$
\langle u,v\rangle_{\Sigma_{A,V}}:=\Re\int_{\R^3}\big((i\nabla + A(x))u\cdot\overline{(i\nabla + A(x))v} + V(x)u\bar{v} + u\bar{v}\big)
$$
$$\|u\|^2_{\Sigma_{A,V}}:=\|u\|^2_{\dot{\Sigma}_{A,V}} + \|u\|^2_2,
\ \mbox{
where 
}\ 
\|u\|^2_{\dot{\Sigma}_{A,V}}:=\|(i\nabla + A(x))u\|^2_2 + \int_{\R^3} V(x)|u|^2.
$$

%The study of uniqueness will need the limiting case $A=0$, where the following Hilbert space will be involved:

% \item[-)]  $ {{\Sigma}}_{0,V}:=\Big\{u\in H^1(\R^3,\R)\ :\ \int_{\R^3}|\nabla u|^2 + (V(x)+1)|u|^2<\infty \Big\}$,
% with inner product and norm
% \[
% \langle u,v\rangle_{{{\Sigma}}_{0,V}} = \int_{\R^3}\nabla u\nabla v + (V(x)+1)uv,
% \qquad  \|u\|_{{{\Sigma}}_{0,V}} = \sqrt{\langle u,u\rangle_{\Sigma_{0,V}}}.
% \]
% %Sometimes, the solution $u$ of \eq{eq1.1} may be a real function times a constant phase, in which case the $u$ is indeed an element of the following Hilbert space $\hat{\Sigma}_{A,V}(\R^N)$ defined by

\item[-)] 
$ \hat{\Sigma}_{A,V}:=\Big\{u\in H^1(\R^3,\C)\ :\ \|u\|^2_{\ddot{\Sigma}_{A,V}}+\|u\|^2_2<\infty \Big\}$,
with inner product and norm 
\[
\langle u,v\rangle_{\hat{\Sigma}_{A,V}} = \int_{\R^3}\big(\nabla u\overline{\nabla v} + (V(x)  + |A(x)|^2+1)u\bar{v} \big),
\qquad  \|u\|_{\hat{\Sigma}_{A,V}} = \sqrt{\langle u,u\rangle_{\hat{\Sigma}_{A,V}}}.
\]
Here $\|u\|^2_{\ddot{\Sigma}_{A,V}} = \int_{\R^3}|\nabla u|^2 + (V(x) + |A(x)|^2)|u|^2$.
\medskip

\item[-)] $\Sigma_{A,V}^\Re$ and $\hat{\Sigma}_{A,V}^\Re$ are the subspaces of real valued functions in $\Sigma_{A,V}$ and $\hat{\Sigma}_{A,V}$, respectively.

\medskip

\item[-)] The symbols ``$\rightarrow$" and ``$\rightharpoonup$" are used to signify strong and weak convergence, respectively, within the related functional spaces.

\medskip

\item[-)] $o_c(1)$ and $O_c(1)$ mean $|o_c(1)|\to 0$ and $|O_c(1)|\le C<+\infty$, as $c\to 0^+$, respectively. 

\end{itemize}

}%end of small

\end{notations}

\medskip

Next, we collect some preliminary results. 
Firstly, we give the Gagliardo-Nirenberg inequality (see \cite{Weinstein-CMP-1983}).

\begin{lemma}\label{le2.1}
Let $N\ge 2$ and $p\in(2,2^*)$ with $2^* = \frac{2N}{N-2}$ if $N\ge 3$ and $2^* =\infty$ if $N=2$. 
Then
\begin{equation*}
\|u\|_p\le C_{N,p}\|\nabla u\|^{\delta_p}_2\|u\|^{(1-\delta_p)}_2,\ \forall u\in H^1(\R^N,\R),
\end{equation*}
where $\mathcal{C}_{N,p} = \Big(\frac{p}{2\|W_p\|^{p-2}_2}\Big)^{\frac{1}{p}},\ W_p$ is the ground state  of $-\Delta W + (\frac{1}{\delta_p}-1)W=\frac{2}{p\delta_p}|W|^{p-2}W$ and $\delta_p = \frac{N(p-2)}{2p}$.
\end{lemma}

In order to study the magnetic term, we need the following diamagnetic inequality, whose proof can be found in \cite[Theorem 7.21]{Lieb-Loss-AMS}.
\begin{lemma}\label{lle2.2}
Let $N\ge 2$, $A\in L^2_{\loc}(\R^N,\R^N)$ and $u\in \Sigma_{A,V}$, then $|u|\in H^1(\R^N,\R)$ and
$$
|\nabla|u|(x)|\le |(i\nabla + A(x))u(x)|,\ \text{a.e.}\ x\in \R^N.
$$

\end{lemma}

\begin{remark}\label{re2.4}
Combing this lemma with Lemma \ref{le2.1}, we immediately have
\begin{equation*}
\|u\|_p\le C_{N,p}\|(i\nabla + A(x))u\|^{\delta_p}_2\|u\|^{1-\delta_p}_2,\ \forall u\in H^1(\R^N,\C),
\end{equation*}
where $C_{N,p}$ and $\delta_p$ are the same as in Lemma \ref{le2.1}.
\end{remark}

According to the Cauchy inequality and the condition \eq{V1}, one can check easily the following operator inequality.
\begin{proposition}\label{pr2.5}
Let $V$ and $A$ satisfy the assumption \eq{V1}. Then   there exist $C_\alpha,C>0$ such that
\begin{eqnarray}
\label{924}& C_{\alpha}\|u\|^2_{\ddot{\Sigma}_{A,V}} \le \|u\|^2_{\dot{\Sigma}_{A,V}}\le 2\|u\|^2_{\ddot{\Sigma}_{A,V}},&  \hspace{1cm} \forall u\in \Sigma_{A,V},
\\
 &  \|u\|^2_2\le C\|u\|^2_{\dot{\Sigma}_{A,V}}& \hspace{1cm} \forall u\in \Sigma_{A,V},\nonumber
\end{eqnarray}
here $C_{\alpha} = \max\limits_{\va\in(0,1),\si\in(0,1)}\min\{1-\va,1-\si,1+\si\alpha-1/\va\} = 1-1/\sqrt{\a+1}$.
 In particular, $\|\cdot\|_{{\Sigma}_{A,V}}$, $\|\cdot\|_{\dot{\Sigma}_{A,V}}$, $\|\cdot\|_{\ddot{\Sigma}_{A,V}}$ are equivalent norms in $\Sigma_{A,V}$.
In addition, the embedding $\Pi: {\Sigma}_{A,V}\hookrightarrow L^2(\R^3)$ is compact.
\end{proposition}
\begin{proof}
The right-hand side in \eq{924} is obvious. For the left-hand side, by the Cauchy inequality and the assumption \eq{V1}, for $\va,\si\in(0,1)$, we have
\begin{align*}
\|u\|^2_{\dot{\Sigma}_{A,V}}\ge \int_{\R^3}(1-\va)|\nabla u|^2 + (1-\si)V(x)|u|^2 + \big(1+\si\alpha-\tfrac{1}{\va}\big)|A(x)|^2|u|^2.
%\ge C_{\alpha}\|u\|^2_{\hat{\Sigma}_{A,V}}.
\end{align*}
Then \eq{924} follows by letting $\va = \si = 1/\sqrt{\a+1}$.

Now, let $R>0$ be such that $V\ge 1$ on $\R^3\setminus \{|x|\le R\}$.
Then, by using the Sobolev inequality, and diamagnetic inequality, 
\begin{equation*}
\begin{split}
\|u\|_2^2 &=\int_{\{|x|\le R\} }|u|^2+\int_{\R^3\setminus \{|x|\le R\}}|u|^2\le c_1 \|\nabla |u|\|_2^2+\int_{\R^3}V(x) |u|^2\\
& \le c_1\|(i\nabla + A(x))u\|^2_2 +\int_{\R^3}V(x) |u|^2
\le C \|u\|^2_{\dot{\Sigma}_{A,V}}.
\end{split}
\end{equation*}
Finally, taking into account that $H^1( K)$ is compactly embedded in $L^2(K)$, for every compact set $K\subset \R^3$, that $V(x)\to\infty$, as $|x|\to\infty$, and arguing as before, we get that $\Pi: {\Sigma}_{A,V}\hookrightarrow L^2(\R^3)$ is compact. 
\end{proof}

Let us consider on $\Sigma_{A,V}$ %$(\Sigma_{A,V},\|\cdot\|_{\dot{\Sigma}_{A,V}})$ 
the magnetic Schr\"odinger operator
%\todogr[inline]{Here I have inserted both the space and the norm, because in proof the $\dot{\Sigma}_{A,V}$-norm appear. In the following I'll see if on $\Sigma_{A,V}$ the $\dot{\Sigma}_{A,V}$-norm will be sufficient in all the paper. In such case we can introduce only this norm. I have to work more on this point.}
\beq
\label{1441}
\hat{L}_{A,V}:=(i\nabla + A(x))^2 + V(x).
\eeq

We have the following spectral analysis for $\hat{L}_{A,V}$:

\begin{lemma}\label{Ale2.8}
Let $A$ and $V$ satisfy \eq{V1}. Then there hold:

(i) the eigenvalues of $\hat{L}_{A,V}$ are real;

(ii) if we repeat each eigenvalue $\hat{\la}_{k,A,V}$, $k\in\N^+$,  according to its (finite) multiplicity, we have
$$
0< \hat{\la}_{1,A,V}\le \hat{\la}_{2,A,V}\le\cdots
$$
and
$$
\hat{\la}_{k,A,V}\to\infty\ \ \text{as}\ k\to\infty;
$$

(iii) there exists an orthonormal basis $\{\hat{\psi}_{k,A,V}\}$ of $L^2(\R^3,\C)$, where $\hat{\psi}_{k,A,V}\in {\Sigma}_{A,V}$ is an eigenfunction corresponding to $\hat{\la}_{k,A,V}$:
$$
\hat{L}_{A,V}\hat{\psi}_{k,A,V} = \hat{\la}_{k,A,V}\hat{\psi}_{k,A,V},\quad \forall k\in \N^+.
$$
\end{lemma}

 The presence of the imaginary unit $i$ causes some difference from the standard arguments in spectral analysis (see, f.i. Section 6.5 in \cite{L.C.Evans-1998}).
 For the sake of completeness, we provide a detailed proof of Lemma \ref{Ale2.8}, in Appendix A for continuity.

\smallskip
 
Now, consider on $\hat \Sigma_{A,V}^\Re$ the Schr\"odinger operator
\[
L_{A,V}:=-\Delta + \big(V(x)+|A(x)|^2).
\]
For $L_{A,V}$ we have an analogous spectral analysis:

\begin{lemma}\label{Ale2.7}
Let $A$ and $V$ satisfy \eq{V1}. Then there hold:

(i) the (real) eigenvalues of $L_{A,V}$ have finite multiplicity;

(ii) if we repeat each eigenvalue $\la_{k,A,V}$, $k\in\N^+$, according to its multiplicity, we have
$$
0<  \la_{1,A,V}\le \la_{2,A,V}\le\cdots
$$
and
$$
\la_{k,A,V}\to\infty\ \ \text{as}\ k\to\infty;
$$

(iii) there exists an orthonormal basis $\{\psi_{k,A,V}\}$ of $L^2(\R^3,\R)$, where $\psi_{k,A,V}\in \hat{\Sigma}^{\Re}_{A,V} $ is an eigenfunction corresponding to $\la_{k,A,V}$:
$$
L_{A,V}\psi_{k,A,V} = \la_{k,A,V}\psi_{k,A,V},\quad \forall k\in\N^+.
$$

(iv) the first eigenvalue $\la_{1,A,V}$ is simple.

\end{lemma}

The standard proof is analogous to the proof of Lemma \ref{Ale2.8} and will be omitted. 
In particular, for the simplicity of the first eigenvalue, we refer the reader to \cite[Theorem 11.8]{Lieb-Loss-AMS}. 

 \smallskip
 
If $A=0$, it is easy to see that the simplicity of the first eigenvalue in the magnetic case is inherited by the corresponding real-valued case. 
Anyway, we state the following result for future reference and outline the proof for the sake of completeness.

\begin{lemma}
\label{LBar}
Let $V$ satisfy \eqref{V1}, then the magnetic operator $\hat L_{0,V}$ has the same eigenevalues of the
corresponding real Schr\"odinger operator $L_{0,V}$, and the same eigenvectors, up to a phase factor. 
In particular, $\hat\lambda_{1,0,V}=\lambda_{1,0,V}$ is simple.
\end{lemma}
\begin{proof}[Sketch of the proof.]
First, it is trivial to observe that the eigenvalue of $\{\lambda_k\}:=\{\lambda_{k,0,V}\}$ of  $L:=L_{0,V}$ are also eigenvalue for $\hat L:=\hat L_{0,V}$.

Now, let us assume that $\bar\lambda=a+bi$, $a,b\in\R$ is an eigenvalue of $\hat L$, and $u=u_1+iu_2\neq 0$, $u_1,u_2\in L^2(\R^3,\R)$, a corresponding eigenvector.
We decompose $u_1,u_2$ in the orthonormal basis $\{e_k\}:=\{\psi_{k,0,V}\}$, namely: $u_1=\sum_kt_k e_k$, $u_2=\sum_ks_k e_k$.
If $\hat L(u_1+i u_2)=(a+b i)(u_1+iu_2)$ and we write it in the components, taking into account that $\hat L e_k=L e_k=\lambda_k e_k$, $k\in \N^+$, we obtain 
$$
\left\{\begin{array}{l}(a-\lambda_k)t_k-b s_k=0\\
b t_k+(a-\lambda_k)s_k=0\end{array}\right.\qquad \forall k\in \N^+.
$$
If $t_{\bar k}, s_{\bar k}\neq 0$ for some $\bar{k}\in N^+$, then $(a-\lambda_{\bar k})^2+b^2=0$, that is $a=\lambda_{\bar k}$, $b=0$. 
So, the eigenvalues are the same.
Finally, if $u=u_1+i u_2$ satisfies $\hat L u=\lambda_1 u$, then $u_1=\bar t e_1$ and $u_2=\bar s e_1$, for suitable $\bar t,\bar s\in \R$, so $u=(\bar t+\bar si)e_1$.
\end{proof}

In general, we cannot obtain the simplicity of the first eigenvalue in the lemma \ref{Ale2.8}. 
However, this property can be obtained, by the implicit function theorem, when $A$ is small in some sense (see Proposition \ref{leA.1}). 
Notice that the simplicity of the first eigenvalue is crucial in Theorem \ref{th1.3}.

\medskip

It is not difficult to obtain from Lemma \ref{Ale2.8} a sufficient and necessary condition, to ensure ${S_c}\cap B(r)\ne\emptyset$.

\begin{lemma}\label{le2.7}
Denote $\hat{\psi}_{c,\tau}=\sqrt{c}\tau^{\frac{3}{2}}{\hat{\psi}}_{1,A,V}(\tau\, \cdot )$, $c,\tau>0$. 
There hold
\begin{equation}\label{AAeq2.4}
\|\hat{\psi}_{c,\tau}\|^2_2=c,\ \ \|\hat{\psi}_{c,1}\|^2_{\dot{\Sigma}_{A,V}} = c\hat{\la}_{1,A,V}\ \text{and}\ \|\hat{\psi}_{c,\tau}\|^2_{\dot{\Sigma}_{A,V}}\to\infty\ \text{as}\ \tau\to \infty;
\end{equation}
moreover, ${S_c}\cap B(r)\ne\emptyset$ if and only if $c\le \frac{r}{\hat{\la}_{1,A,V}}$.
\end{lemma}
\begin{proof}
The first relations in \eqref{AAeq2.4} came from a direct computation, while the third is easily seen also taking into account \eq{924}.

Next, if $c\le \frac{r}{\hat{\la}_{1,A,V}}$, then
\[
\|\hat{\psi}_{c,1}\|^2_{\dot{\Sigma}_{A,V}} = c\|\hat{\psi}_{1,A,V}\|^2_{\dot{\Sigma}_{A,V}}\le r,
\text{ and } \hat{\psi}_{c,1}\in S_c\cap B(r).
\]
On the other hand, if $S_c\cap B(r)\ne\emptyset$, then, letting $u\in S_c\cap B(r)$, by the variational characterization of the first eigenvalue we have
$$
r\ge \|u\|^2_{\dot{\Sigma}_{A,V}} = \|u\|^2_2 \cdot  \Big\|\frac{u}{\|u\|_2}\Big\|^2_{\dot{\Sigma}_{A,V}}\ge c\, \hat{\la}_{1,A,V},
$$
which implies $c\le \frac{r}{\hat{\la}_{1,A,V}}$.
\end{proof}

%%%%%%%%%%%%%%%%%%%%%%%%%%%%%%%%%%%%%%%%%%%%%%%%%%%%%%%%%%%%%%%%%%%

\section{Local minimum and ground state solution, proof of Theorem \ref{th1.1}}
\label{S3}

%%%%%%%%%%%%%%%%%%%%%%%%%%%%%%%%%%%%%%%%%%%%%%%%%%%%%%%%%%%%%%%%%%%

In this section, we give the proof of Theorem \ref{th1.1}. 
To begin with, we show that $\mathcal{E}|_{{S_c}}$ presents a local minimum structure, which guarantees that the minimizers of $m^r_{{c}}$ are indeed critical points of $\mathcal{E}|_{{S_c}}$.

\medskip

\begin{proposition}\label{pr3.1}
Let $A$ and $V$ satisfy the condition \eq{V1}. 
For any $r>0$, if $\mu,\sigma>0$ satisfy $0<\sigma<\mu< 1$,  $c$ is such that
\beq\label{1725}
0<{c} <\Big(\frac{2}{\sqrt{r}}\frac{\mu-\si}{\mu}\frac{1}{C^4_{3,4}}\Big)^2,
\eeq
 ${S_c}\cap\big(B(r)\backslash B(\mu r)\big) \ne\emptyset$ and ${S_c}\cap  B(\sigma r)  \ne\emptyset$,  then
\begin{equation*}
\inf_{u\in {S_c}\cap B(\sigma r)}\mathcal{E}(u)<\inf_{u\in {S_c}\cap (B(r)\backslash B(\mu r))}\mathcal{E}(u),
\end{equation*}
where $C_{3,4}$ is the Gagliardo-Nirenberg constant given in Lemma \ref{le2.1}.
\end{proposition}

\begin{proof}
By assumption, we can choose $u_{\mu}\in {S_c}\cap(B(r)\backslash B(\mu r))$ and $u_{\sigma}\in {S_c}\cap B(\sigma r)$. According to Remark \ref{re2.4}, we have
\begin{equation*}
% \nonumber to remove numbering (before each equation)
 \begin{split}
\mathcal{E}(u_{\mu})
%& = \frac{1}{2}\|u_{\mu}\|^2_{\dot{\Sigma}_{A,V}} - \frac{1}{4}\int_{\R^3}|u|^4\\
& \ge \frac{1}{2}\|u_{\mu}\|^2_{\dot{\Sigma}_{A,V}} - \frac{\sqrt{c}}{4}C^4_{3,4}\|u_{\mu}\|^{3}_{\dot{\Sigma}_{A,V}}
\ge \|u_{\mu}\|^2_{\dot{\Sigma}_{A,V}}\Big(\frac{1}{2}-\frac{\sqrt{c}}{4}C^4_{3,4}
\sqrt{r}\Big)\\
&\quad\ge \mu r\Big(\frac{1}{2}-\frac{\sqrt{c}}{4}C^4_{3,4}
\sqrt{r}\Big).
\end{split}
\end{equation*}
Obviously,
$
\mathcal{E}(u_{\sigma})
\le \frac{1}{2}\|u_{\sigma}\|^2_{\dot{\Sigma}_{A,V}}\le \frac{1}{2}\sigma r.
$
Then  
$$
\mathcal{E}(u_{\sigma})<\mathcal{E}(u_{\mu}),
$$
because of \eqref{1725}.
This completes the proof.
\end{proof}

The following Pohozaev identity is crucial in proving that a local minimizer in ${S_c}\cap B(r)$ is also an energy ground state in ${S_c}\cap \Sigma_{A,V}$.  
Actually, it provides a key relation that does not involve the unknown Lagrange arising from the mass constraint; see \eq{ueq3.25} below, for example. 
Considering the more general cases $N =2,3$, and a generic subcritical nonlinearity $p$, with $2<p<2^*$, we have

\begin{proposition}\label{pr3.2}
Let $N=2,3$, $2<p<2^*$, $A,V\in C^{1,\alpha}(\R^N\backslash\{0\})$ satisfy  \eqref{V1} and \eq{V2}$(a)$, and $\la\in\R$. 
If $v\in{\Sigma}_{A,V}$ weakly solves
\begin{equation}\label{peq3.6}
(i\nabla + A(x))^2v + V(x)v- |v|^{p-2}v = \la v,
\end{equation}
then the following Pohozaev identify
\begin{equation*}
\begin{split}
\mathcal{P}(v):=\int_{\R^N}|\nabla v|^2  - \tfrac{N(p-2)}{2p}\int_{\R^N}|v|^p + T_{A,V}(v)=0
\end{split}
\end{equation*}
holds, where
\begin{equation}
\label{1505}
\begin{split}
T_{A,V}(v):&=  \Re\int_{\R^N}i(A(x)\cdot\nabla v)\bar{v} - \Re\int_{\R^N}i\Big(\sum_{j=1}^N\big(x\cdot \nabla A_j(x)\big)\partial_{x_j} v\Big)\bar{v} \\
& -\frac{1}{2}\int_{\R^N}\big(x\cdot \nabla V(x)\big) |v|^2 - \int_{\R^N}\Big(\sum_{j=1}^N A_j(x)\big(x\cdot\nabla A_j\big)\Big) |v|^2.
\end{split}
\end{equation}

\end{proposition}
\begin{proof}
Note that $v\in C^{2,\alpha'}(\R^N\backslash\{0\})$ for some $\alpha'\in(0,1)$ by Remark \ref{Rreg}. 
Note also that the potential and magnetic field may also be singular at the origin. 
Therefore, applying truncation techniques is essential in the proof. 
Let $\bar\eta\in C^{\infty}(\R^N,[0,1])$  be a smooth function satisfying $\bar\eta = 1$ in $\R^N\backslash B_{2}(0)$ and $\bar\eta = 0$ on $B_1(0)$, and let $\eta_\va:=\bar{\eta}(\cdot/\va)(1-\bar{\eta}(\va\cdot))$. 

We have the following computation:
\begin{align*}
%\label{abeq3.2}
\begin{split}
    \Re\int_{\R^N}(-\Delta v)\big(\eta_\va x\cdot\nabla \bar{v}\big)&=\frac{2-N}{2}\int_{\R^N}\eta_\va |\nabla v|^2 - \frac{1}{2}\int_{\R^N}(x\cdot \nabla \eta_\va)|\nabla v|^2 
   \\ &\quad +\int_{\R^N} (\nabla v\cdot\nabla \eta_\epsilon)(x\cdot\nabla\bar v),
    \\[4pt]
     \Re\int_{\R^N}V(x)v (\eta_\va x\cdot \nabla \bar{v}) &= -\frac{N}{2}\int_{\R^N}  \eta_\va V(x)|v|^2-\frac{1}{2}\int_{\R^N}\eta_\va \big(x\cdot \nabla V(x)\big) |v|^2\\
     &\quad  - \frac{1}{2}\int_{\R^N}(x\cdot \nabla\eta_\va)V(x)|v|^2 , \\[4pt]
    \Re\int_{\R^N}|A(x)|^2v(\eta_\va x\cdot \nabla \bar{v})&= -\frac{N}{2}\int_{\R^N}  \eta_\va|A(x)|^2|v|^2-\sum_{i=1}^N\int_{\R^N}\eta_\va A_i(x)\big(x\cdot\nabla A_i(x)\big) |v|^2\\
    &\quad - \frac{1}{2}\int_{\R^N}(x\cdot \nabla \eta_\va)|A(x)|^2|v|^2,
\end{split}
\end{align*}

\begin{align*}
 &\quad\Re\int_{\R^N}i\big(A(x)\cdot\nabla v\big)(\eta_\va x\cdot \nabla \bar{v})  \\
& = \Re\int_{\R^N}i\, \eta_\va\big(A(x)\cdot \nabla v\big)\frac{d}{dt}\big(\bar{v}(tx)\big)_{t=1}\\
%& = \Re\int_{\R^N}i\big(A(x)\cdot \nabla v\big)\frac{d}{dt}\big(\bar{v}_t\big)_{t=1}\\
%& = \Re \frac{d}{dt}\Big(\int_{\R^N}i\big(A(x)\nabla v\big)\bar{v}_t\Big)_{t=1}\\
& = \Re \frac{d}{dt}\Big(\frac{1}{t^N}\int_{\R^N}i\, \eta_\va(x/t)\big(A(x/t)\cdot\nabla v(x/t)\big)\bar{v}(x)\Big)_{t=1}\\
 &= -N\Re \int_{\R^N}i\eta_\va \big(A(x)\cdot\nabla v \big)\bar{v} 
- \Re\int_{\R^N}i\Big(\sum_{j=1}^N\eta_\va (x\cdot \nabla A_j(x))\partial_{x_j}v\Big)\bar{v}\\
&\quad - \Re\int_{\R^N}i\eta_\va\sum_{j=1}^NA_j(x) (x\cdot \nabla \partial_{x_j}  v) \bar{v} - \Re\int_{\R^N}i(x\cdot\nabla\eta_\va) (A(x)\cdot \nabla v)\bar{v},
\end{align*}
where
\begin{align*}
&\quad \Re\int_{\R^N}i\eta_\va \sum_{j=1}^N A_j(x)( x\cdot \nabla \partial_{x_j} v) \, \bar{v}\\
&= -\Re\int_{\R^N}i\eta_\va(x\cdot\nabla v)(\text{div} A(x))\bar{v} - \Re\int_{\R^N}i\eta_\va(A(x)\cdot\nabla v)\bar{v}\\
&\quad  - \Re\int_{\R^N}i\eta_\va(A(x)\cdot\nabla\bar{v})(x\cdot \nabla v)- \Re\int_{\R^N}i(A(x)\cdot\nabla \eta_\va)(x\cdot \nabla v)\bar{v}\\
&= \Re\int_{\R^N}i\eta_\va(\text{div} A(x))v(x\cdot\nabla \bar{v}) { -}\Re\int_{\R^N}i\eta_\va(A(x)\cdot \nabla {v})\bar v\\
&+ \Re\int_{\R^N}i\eta_\va(A(x)\cdot\nabla{v})(x\cdot \nabla \bar{v})  - 
\Re\int_{\R^N}i (A(x)\cdot\nabla  \eta_\va)( x\cdot \nabla v )\bar v.
\end{align*}

Then
\begin{equation*}
%\label{abeq3.3}
\begin{split}
 &\quad 2\Re\int_{\R^N}i\big(A(x)\cdot \nabla v\big)(\eta_\va x\cdot \nabla \bar{v})\\
&=-N\Re \int_{\R^N}i\eta_\va \big(A(x)\cdot\nabla v(x)\big)\bar{v}(x)- \Re\int_{\R^N}i\eta_\va \Big(\sum_{j=1}^N(x\cdot \nabla A_j(x))\partial_{x_j}v\Big)\bar{v}\\
&\quad - \Re\int_{\R^N}i\eta_\va(\text{div} A(x))v (x\cdot\nabla \bar{v})   +   \Re\int_{\R^N}i\eta_\va(A(x)\cdot\nabla v)\bar{v}\\
&\quad -  \Re\int_{\R^N}i(x\cdot\nabla \eta_\va)(A(x)\cdot\nabla v)\bar{v}
+ \Re\int_{\R^N}i(A(x)\cdot\nabla \eta_\va)( x \cdot\nabla {v})\bar v.
\end{split}
\end{equation*}
%\todobl[inline]{It seems to me that \eq{abeq3.3} has to be fixed. I do not do that because I can be wrong.}
Obviously,  $|x\nabla\eta_\va|\le C$,  $x\in \R^3$ and   $\text{supp}\nabla \eta_\va \subset \big( B_{2\va}(0) \setminus B_{\va}(0) \big) \cup  \big(B_{\frac{2}{\va}}(0) \setminus B_{\frac{1}{\va}(0) } \big)$ .  
Hence, since $v\in \Sigma_{A,V}$, we conclude
\begin{equation*}
    %\label{943}
\begin{split}
\int_{\R^N}(x\cdot\nabla\eta_\va)|\nabla v|^2, \,\,\int_{\R^N}(x\cdot \nabla\eta_\va)|A(x)|^2|v|^2= o_\va(1) ; \\ \Re\int_{\R^N}i(x\cdot\nabla \eta_\va)(A(x)\cdot\nabla v)\bar{v}, \,\,
 \Re\int_{\R^N}i(A(x)\cdot\nabla \eta_\va)(x \cdot\nabla v)\bar{v} 
= o_\va(1).
\end{split}
\end{equation*}

Moreover, also taking into account \eq{V2}, we have
\begin{align*}
%\label{abeq3.4}
\begin{split}
\int_{\R^N}\eta_\va \big(x\cdot \nabla V(x)\big)|v|^2&= \int_{\R^N} \big(x\cdot \nabla V(x)\big)|v|^2 + o_\va(1),\\  \sum_{i=1}^N\int_{\R^N}\eta_\va A_i(x\cdot\nabla A_i) |v|^2 &= \sum_{i=1}^N\int_{\R^N} A_i(x\cdot\nabla A_i) |v|^2 + o_\va(1)
\end{split}
\end{align*}
and

\beq
\label{ueq3.8}
\begin{split}
 &\quad 2\Re\int_{\R^N}i\big(A(x)\cdot\nabla v\big)(\eta_\va x\cdot \nabla \bar{v})\\
&=-N\Re \int_{\R^N}i\big(A(x)\cdot\nabla v\big)\bar{v} - \Re\int_{\R^N}i\Big(\sum_{j=1}^N(x\cdot \nabla A_j(x))\partial_{x_j}v\Big)\bar{v}\\
&\quad - \Re\int_{\R^N}i(\text{div} A(x)) v (x\cdot\nabla \bar{v})  + \Re\int_{\R^N}i(A(x)\cdot\nabla v)\, \bar{v}+o_\va(1).
\end{split}
\eeq

 Taking into account that $\partial_{x_i}|v|^p=p|v|^{p-2}\Re (v\, \bar v_{x_i})$, by similar computations we get
$$
%\label{1401}
\Re \int_{\R^N}|v|^{p-2}v\, (\eta_\va x\cdot \nabla \bar v)=-\frac N p \int_{\R^N}|v|^p+o_\va (1),
$$
$$
\Re \int_{\R^N} v\, (\eta_\va x\cdot \nabla \bar v)=-\frac N 2 \int_{\R^N}|v|^2+o_\va (1).
$$
Testing \eq{peq3.6}  with $\eta_\va x\cdot \nabla \bar{v}$, by the definition of $(i\nabla + A(x))^2$ and the computations above  we find
\beq
\label{eq3.11}
\begin{split}
&\quad\frac{2-N}{2}\int_{\R^N}|\nabla v|^2 - \frac{N}{2}\int_{\R^N}|A(x)|^2|v|^2 - \frac{N}{2}\int_{\R^N}V(x)|v|^2 + \frac{N}{p}\int_{\R^N}|v|^p\\
&\quad-\frac{1}{2}\int_{\R^N}\big(x\cdot \nabla V(x)\big) |v|^2 - \int_{\R^N}\Big(\sum_{i=1}^NA_i(x)\big(x\cdot \nabla A_i(x)\big)\Big)|v|^2\\
&\quad + 2\Re\int_{\R^N}i\big(A(x)\cdot \nabla v\big)(  \eta_\va x\cdot \nabla \bar{v})+\Re\int_{\R^N}i(\text{div} A(x))v(x\cdot\nabla \bar{v})\\
& = -\frac{N}{2}\la \int_{\R^N}|v|^2 + o_\va(1).
\end{split}
\eeq
On the other hand,  multiplying \eq{peq3.6} with $\bar{v}$, we find
\beq
\label{eq3.12}
\int_{\R^N}|\nabla v|^2 + \int_{\R^N}|A(x)|^2|v|^2 + \int_{\R^N}V(x)|v|^2 - \int_{\R^N}|v|^p + 2\Re\int_{\R^N}i (A(x)\cdot \nabla v)\bar{v}
= \la\int_{\R^N}|v|^2.
% \begin{split}
% &\quad \int_{\R^N}|\nabla v|^2 + \int_{\R^N}|A(x)|^2|v|^2 + \int_{\R^N}V(x)|v|^2 - \int_{\R^N}|v|^p + 2\Re\int_{\R^N}i\nabla v A(x)\bar{v}\\
% & = \la\int_{\R^N}|v|^2.
% \end{split}
\eeq
Thus, from \eq{eq3.11} and \eq{eq3.12}, it follows
\begin{equation}
\label{1533}
\begin{split}
 o_\va (1)=& \int_{\R^N}|\nabla v|^2  + \frac N2\Big(\frac{2-p}{p}\Big)\int_{\R^N}|v|^p+N\Re\int_{\R^N}i(A(x)\cdot\nabla v) \bar{v} \\
&+  2\Re\int_{\R^N}i\big(A(x)\cdot \nabla v\big)\big(\eta_\va) x\cdot \nabla \bar{v}\big)+  \Re\int_{\R^N}i(\text{div}A)v(x\cdot\nabla \bar{v})\\
& - \Big(\frac{1}{2}\int_{\R^N}\big(x\cdot \nabla V(x)\big) |v|^2 +  \int_{\R^N}\Big(\sum_{i=1}^N A_i(x)\big(x\cdot\nabla A_i\big)\Big) |v|^2 \Big).
\end{split}
\end{equation}
Finally, inserting \eq{ueq3.8} in \eqref{1533}  and letting $\va\to 0$, we complete the proof.
% \begin{equation*}
% \begin{split}
% &\quad \int_{\R^N}|\nabla v|^2  - \delta_p\int_{\R^N}|v|^p + N\Re\int_{\R^N}i\nabla v A(x)\bar{v} +  2\Re\int_{\R^N}i\big(A(x)\nabla v\big)x\cdot \nabla \bar{v}\\
% &\quad + \Re\int_{\R^N}i(\text{div}A)v(x\cdot\nabla \bar{v}) - \Big[\frac{1}{2}\int_{\R^N}x\nabla V(x) |v|^2 + \int_{\R^N}\Big(\sum_{i=1}^N A_i(x)\big(x\cdot\nabla A_i\big)\Big) |v|^2 \Big]\\
% &= \int_{\R^N}|\nabla v|^2  - \delta_p\int_{\R^N}|v|^p + \Re\int_{\R^N}i(A(x)\cdot\nabla v)\bar{v}- \Re\int_{\R^N}i\Big(\sum_{j=1}^N(x\nabla A_j(x))\partial x_j v\Big)\bar{v}\\
% &\quad - \Big[\frac{1}{2}\int_{\R^N}x\nabla V(x) |v|^2 + \int_{\R^N}\Big(\sum_{i=1}^N A_i(x)\big(x\cdot\nabla A_i\big)\Big) |v|^2 \Big]\\
% & = 0.
% \end{split}
% \end{equation*}
% This completes the proof.
\end{proof}

\begin{example}%\label{re3.3}
 In the case of nonlinear Schr\"{o}dinger equation with rotation, where $A(x) = \rho(-x_2,x_1,0)$ and  $V(x) = |x|^2 - \rho^2 (x^2_1+x^2_2)$, it turns out $T_{A,V}(v) = -\int_{\R^3} |x|^2|v|^2$, so the Pohozaev identity takes the form
$$
\mathcal{P}(v) = \int_{\R^3}|\nabla v|^2 - \frac{3}{4}\int_{\R^3}|v|^4 - \int_{\R^3}|x|^2|v|^2 = 0
$$
(see also \cite{X.Luo-T.Yang-JDE-2020}).
\end{example}

We are now in a position to give the proof of Theorem \ref{th1.1}.

\begin{proof}[\textsf{1. Proof of part 1 of Theorem \ref{th1.1}: Existence of local minimizer.}] 
For $(\si,r)\in(0,1)\times (0,\infty)$, denote
\begin{equation*}
%\label{neq1.4}
F_{A,V}(\si,r) := \min\left\{\frac{\si r}{\hat{\la}_{1,A,V}},\frac{1}{C^8_{3,4}}\frac{4}{r}(1-\si)^2\right\},
\end{equation*}
and let us set
\begin{equation}
\label{1820}
c_{A,V}(r):=\max_{\si\in(0,1)}F_{A,V}(\si,r)=F(\si_r,r)
\end{equation}
for a suitable $\si_r\in(0,1)$ (see \eq{1549}).
Now, consider
\[
0<c<c_{A,V}(r).
\]
Clearly, $S_c\cap B(\si_r r)\ne \emptyset$ by Lemma \ref{le2.7} and there exists $\mu_r\in(0,1)$ such that
\begin{equation}\label{adeq3.5}
0<c<\Big(\frac{1}{C^4_{3,4}}\frac{2}{\sqrt{r}}\frac{\mu_r-\si_r}{\mu_r}\Big)^2.
\end{equation}
Then, letting $\hat{\psi}_{c,\tau}$ be the function given in \eq{AAeq2.4},  there exists  $\tau_{\mu_r}>0$  such that
\[\hat{\psi}_{c,\tau_{\mu_r}}\in {S_c}\cap (B(r)\backslash B(\mu_r r)).\]
Hence,
\begin{equation}\label{Aeq3.18}
{S_c}\cap (B(r)\backslash B(\mu_r  r))\ne\emptyset\ \text{and}\ {S_c}\cap B(\si_r  r)\ne\emptyset.
\end{equation}
At this point, the proof is standard. 
Let $(u_n)$ be a minimizing sequence for $m^{r}_{{c}}=\inf_{u\in {S_c}\cap B(r)}\mathcal{E}(u)$. Obviously, $(u_n)$ is bounded in ${\Sigma}_{A,V}$. 
By the compactness embedding (see Proposition  \ref{pr2.5}), there exists $u_{c,A,V}\in{\Sigma}_{A,V}$ such that
$$
\left\{
  \begin{array}{ll}
    u_n\rightharpoonup u_{{c},A,V}, & \text{in}\ {\Sigma}_{A,V} \\
    u_n\to u_{{c},A,V}, & \text{in}\ L^q(\R^3,\C)\ \text{for all}\ q\in[2,6) \\
    u_n\to u_{{c},A,V}, & \text{a.e. in}\ \R^3.
  \end{array}
\right.
$$
Consequently, $u_{{c},A,V}\in {S_c}\cap B(r)$, by the weak lower semi-continuous of $\|\cdot\|_{\dot{\Sigma}_{A,V}}$, and
\begin{equation*}
%\label{ueq3.15}
m^{r}_{{c}}\le \mathcal{E}(u_{{c},A,V})\le \lim\limits_{n\to\infty} \mathcal{E}(u_n)= m^{r}_{{c}}.
\end{equation*}
Thus, $\mathcal{E}(u_{{c},A,V}) = m^{r}_{{c}}$ and $u_n\to u_{{c},A,V}$ in ${\Sigma}_{A,V}$ as $n\to \infty$. This implies that any minimizing sequence for $m^r_{{c}}$ is precompact and $\mathcal{M}^{r}_{{c}}\ne\emptyset$.

Notice that from \eq{Aeq3.18}, \eq{adeq3.5}, and Proposition \ref{pr3.1}, we infer
\beq\label{ntang}
0<c<c_{A,V}(r) \quad \Longrightarrow\quad \sup_{u\in \mathcal{M}^{r}_c}\|u\|_{\dot\Sigma_{A,V}}<r.
\eeq
Then $u_{{c},A,V}$ is indeed a critical point of $\mathcal{E}|_{{S_c}}$. 
So, there exists a Lagrange multiplier $\la_{{c},A,V}\in\R$ such that $(u_{{c},A,V},\la_{{c},A,V})$ is a weak solution to the problem \eq{eq1.1} if $0<c<c_{A,V}(r)$.

\end{proof}

Next, inspired by \cite{Bellazzni-et.al.-CMP-2017}, we show that $u_{{c},A,V}$ is an energy ground state, for small ${c}>0$. 
%The asymptotic behavior of $u_{{c},A,V}$ as $c\to0^+$ and the stability of the set $\mathcal{M}^r_c$ are also studied.
One of the key steps is to use the Pohozaev identity (Proposition \ref{pr3.2}) in \eq{ueq3.25}.
\begin{proof}[\textsf{2. Proof of part 2 of Theorem \ref{th1.1}: local minimizer to energy ground state.}]
In this proof, we consider $0<c<c_{A,V}(r)$.
Denote
\[
A_{m^r_c} = \{v\in{S_c}\, :\,  (\mathcal{E}|_{{S_c}})'(v)=0,\mathcal{E}(v)\le m^r_c\},
\]
and notice that $A_{m^r_c}\neq\emptyset$ because $u_{c,A,V}\in A_{m^r_c}$. 
Clearly,
\begin{equation}\label{adeq3.9}
\inf_{u\in\{v\in {S_c}\, :\,  (\mathcal{E}|_{{S_c}})'(v)=0\}}\mathcal{E}(u) = \inf_{u\in A_{m^r_c}}\mathcal{E}(u).
\end{equation}

We claim that $A_{m^r_c}\subset B(r)$ if $c>0$ is suitable small.
Let $v\in{\Sigma}_{A,V}$ be a function satisfying
$$
(\mathcal{E}|_{{S_c}})'(v) = 0.
$$
Then $v$ satisfies
\begin{equation*}
\big((i\nabla + A(x))^2 + V(x)\big)v - |v|^{2}v = \la v,\ x\in\R^3
\end{equation*}
for some $\la\in\R$. 
It follows from Proposition \ref{pr3.2} that
$$
\frac{1}{4}\int_{\R^3}|v|^4=\frac{1}{3}\int_{\R^3}|\nabla v|^2  + \frac{1}{3}\, T_{A,V}(v),
$$
where $T_{A,V}$ is defined in \eqref{1505}.
% \beq
% \begin{split}
% T_{A,V}(v)&= \Re\int_{\R^3}i(A(x)\cdot\nabla v)\bar{v} - \Re\int_{\R^3}i\Big(\sum_{j=1}^3(x\nabla A_j(x))\partial_{x_j} v\Big)\bar{v} \\
% & - \Big[\frac{1}{2}\int_{\R^3}x\nabla V(x) |v|^2 + \int_{\R^3}\Big(\sum_{i=1}^3 A_i(x)\big(x\cdot\nabla A_i\big)\Big) |v|^2 \Big].
% \end{split}
% \eeq
Then
\beq
\label{ueq3.25}
\begin{split}
\mathcal{E}(v)&= \frac{1}{6}\int_{\R^3}|\nabla v|^2 + \frac{1}{2}\int_{\R^3}\big(|A(x)|^2 + V(x)\big)|v|^2 + \Re\int_{\R^3}i (A(x)\cdot\nabla v)\bar{v} - \frac{1}{3}T_{A,V}(v)\\
&= \frac{1}{6}\int_{\R^3}|\nabla v|^2 + \frac{1}{2}\int_{\R^3}\big(|A(x)|^2 + V(x)\big)|v|^2 + \Re\int_{\R^3}i (A(x)\cdot\nabla v)\bar{v}\\
&\quad - \frac{1}{3}\Re\int_{\R^3}i\big(A(x)\cdot \nabla v\big)\bar{v} + \frac{1}{3}\Re\int_{\R^3}i\Big(\sum_{j=1}^3\big(x\cdot\nabla A_j(x)\big)\partial_{x_j} v\Big)\bar{v}\\
&\quad + \frac{1}{3}\Big[\frac{1}{2}\int_{\R^3}\big(x\cdot\nabla V(x)\big)|v|^2 + \int_{\R^3}\Big(\sum_{j=1}^3A_j(x)\big(x\cdot \nabla A_j(x)\big)\Big)|v|^2\Big].
\end{split}
\eeq
Consequently, by the Cauchy inequality, for all $\va>0$, we have
\beq
\label{Aeq3.25}
\begin{split}
\mathcal{E}(v)&\ge\Big(\frac{1}{6}-\frac{\va}{2}\Big)\int_{\R^3}|\nabla v|^2 + \int_{\R^3}\Big[\frac{1}{2}V(x) + \frac{1}{6}\big(x\cdot\nabla V(x)\big) + \Big(\frac{1}{2} |A(x)|^2- \frac{1}{3\va}|A(x)|^2\\
 &\qquad - \frac{1}{6\va}\sum_{j=1}^3\big|\big(x\cdot \nabla A_j(x)\big)\big|^2 + \frac{1}{3}\sum_{j=1}^3A_j(x)\big(x\cdot \nabla A_j(x)\big)\Big)\Big]|v|^2\\
&= \Big(\frac{1}{6}-\frac{\va}{2}\Big)\int_{\R^3}|\nabla v|^2 + \tilde{T}_{A,V,\va}  (x)\, |v|^2
\end{split}
\eeq
(see \eq{Po}).
In particular, letting $\va_0\in\big(0,\frac{1}{3}\big)$ and $\alpha_0$ be the two constants given in the assumption {\eq{V2}}, we can infer from \eq{Aeq3.25} and the assumption \eq{V1} that
$$
 \int_{\R^3} \tilde{T}_{A,V,\va_0}(x)\, |v|^2\,\ge \alpha_0\int_{\R^3}V(x)|v|^2\ge \frac{\alpha_0}{2}\min\{1,{ \bar \a}\}\int_{\R^3}(V(x)+|A(x)|^2)|v|^2.
$$
Note that
\begin{equation}\label{eq5-1}
2m^r_{{c}} \le 2\mathcal{E}(\sqrt{c}\hat{\psi}_{1,A,V})< \|\sqrt{c}\hat{\psi}_{1,A,V}\|^2_{\dot{\Sigma}_{A,V}} = c\hat{\la}_{1,A,V}.
\end{equation}
Therefore, letting $\hat{v}\in A_{m^r_c}$,  by Proposition \ref{pr2.5}, we have
\begin{equation}
\label{Aeq3.26}
\begin{split}
\frac{\hat{\la}_{1,A,V}}{2}c\ge m^r_c\ge \mathcal{E}(\hat{v})
&\ge \min\{\frac{1}{6} - \frac{\va_0}{2},\frac{\alpha_0}{2}\min\{1,{\bar\a}\}\}\|\hat{v}\|^2_{\ddot{\Sigma}_{A,V}}\\
&\ge \frac{1}{2}\min\{\frac{1}{6} - \frac{\va_0}{2},\frac{\alpha_0}{2}\min\{1,{\bar\a}\}\}\|\hat{v}\|^2_{\dot{\Sigma}_{A,V}}.
%\ge \frac{1}{2}\min\{\frac{1}{6} - \frac{\va_0}{2},\frac{\alpha_0}{2}\min\{1,\a\}\}\|v\|^2_{\dot{\Sigma}_{A,V}}
%&\ge \frac{\min\{\frac{1}{6} - \frac{\va_0}{2},\alpha_0\}}{\Lambda_\si} \|v\|^2_{\dot{\Sigma}_{A,V}}
\end{split}
\end{equation}
Hence, letting
\[0<c<\frac{\min\{1 - 3\va_0,3\alpha_0\min\{1,\bar\a\}\}}{6\hat{\la}_{1,A,V}}\cdot r,\]
we have
$\|\hat{v}\|^2_{\dot{\Sigma}_{A,V}}\le r$,
which gives the claim.

Finally, by the claim above, it holds $\inf_{u\in A_{m^r_c}}\mathcal{E}(u)\ge m^r_c$, for every $c$ such that 
 \beq
\label{1517}
0<c<\tilde c_{A,V}(r):=\min\left\{c_{A,V}(r)\, , \,  \frac{\min\{1 - 3\va_0,3\alpha_0\min\{1,\bar\a\}\}}{6\hat{\la}_{1,A,V}}\cdot r\right\}.
\eeq
Then, by \eq{adeq3.9}, we have 
\[
\mathcal{E}(u_{c,A,V}) = m^r_c = \inf\{\mathcal{E}(v)\, :\, v\in S_c,\ (\mathcal{E}|_{S_c})'(v) = 0\},\quad \mbox{ for }\ 0<c<\tilde c_{A,V}(r),
\]
which completes the proof.
\end{proof}

\begin{proof}[\textsf{Proof of Proposition \ref{cr1.4}}.]
Let $c\ge c^* = 2\frac{\hat{\la}_{1,A,V}}{\|\hat{\psi}_{1,A,V}\|^4_4}$ and suppose, by contradiction, that $r>0$ exists so that \eq{eq1.1} has a solution $v$ achieving $m^r_c$. 
 By \eq{Aeq3.26} we have
\begin{equation}\label{adeq3.12}
\mathcal{E}(v) \ge \min\left\{\frac{1}{6} - \frac{\va_0}{2},\frac{\alpha_0}{2}\min\{1,\bar\a\}\right\}\|v\|^2_{\ddot{\Sigma}_{A,V}}>0.
\end{equation}

 By Lemma \ref{le2.7}, it has to be $ r\ge c\hat{\la}_{1,A,V}$.
Then $\sqrt{c}\hat{\psi}_{1,A,V}\in S_c\cap B(r)$ and
\beq
\label{1435}
\inf_{S_c\cap B(r)}\mathcal{E}(u)\le \mathcal{E}(\sqrt{c}\hat{\psi}_{1,A,V}) = \frac{c}{2}\Big(\hat{\la}_{1,A,V} - \frac{c}{2}\|\hat{\psi}_{1,A,V}\|^4_4\Big)\le \frac{c}{2}\Big(\hat{\la}_{1,A,V} - \frac{c^*}{2}\|\hat{\psi}_{1,A,V}\|^4_4\Big) = 0.
\eeq
But \eqref{1435} implies
$$
\mathcal{E}(v) = m^r_c \le 0,
$$
which is a contradiction to \eq{adeq3.12}. 
Thus, the proof is concluded. 
\end{proof}

\medskip

At the end of this section, we prove the stability assertion in Theorem \ref{th1.1}. 
First, we verify that $\mathcal{M}^r_c$ is compact. 
For a $u\in \Sigma_{A,V}$, we denote the flow associated with the initial value $u$ as $\psi_u(t)$.

\begin{lemma}
    \label{Lcomp}
  The minimal set $\mathcal{M}^r_c$ is compact in $\Sigma_{A,V}$.
\end{lemma}
\begin{proof} In fact, if $(u_n)$ is a sequence in $\mathcal{M}^r_c$, then $(\|u_n\|^2_{\dot{\Sigma}_{A,V}})$ is bounded, and, taking into account the compact embedding of $\Sigma_{A,V}$ in $L^2$, we can find $u_0\in \Sigma_{A,V}$ such that $u_n\to u_0$, as $n\to\infty$, weakly in $\Sigma_{A,V}$ and strongly in $L^2$ and in $L^4$.
As a consequence, $u_0\in S_c\cap B(r)$.
then, $u_n$ has to strongly converge to $u_0$ in $\Sigma_{A,V}$, otherwise, also since $\|u_n\|_4\to \|u_0\|_4$, we infer $ m^r_c\le \mathcal{E}({u_0})<\liminf\limits_{n\to\infty}\mathcal{E}({u_n})=m^r_c$.
\end{proof}

\begin{proof}[\textsf{3. Proof of part 3 of Theorem \ref{th1.1}: orbital stability of $\mathcal{M}^r_{{c}}$.}] 
 The first claim is that $\de>0$ exists such that if $\dist (u,\mathcal{M}^r_c)_{\Sigma_{A,V}}<\de$, then $\psi_u(t)$ exists globally in time.

\smallskip

We argue by contradiction. 
Suppose, contrary to our claim, that there exists a sequence $\de_n\to 0$, as $n\to\infty$, and $u_n$ in $\Sigma_{A,V}$ satisfying $\dist (u_{n},\mathcal{M}^r_c)_{\Sigma_{A,V}}<\de_n$ and let the corresponding maximal time $T_{n,\max}$ be finite.
Hence, we must have
\begin{equation}
\label{1821}
\limsup_{t\to T^-_{n,\max}}\|\psi_{u_n}(t)\|^2_{\dot{\Sigma}_{A,V}} = \infty, \qquad \forall n\in\N.
\end{equation}
Indeed, if $\|\psi_{u_n}(t)\|^2_{\dot{\Sigma}_{A,V}}$ is bounded then,  since  $\psi_{u_n}$ solves the evolution equation \eq{keq1.1}, we see that $\frac{d}{d\,t}\psi_{u_n}$ is also bounded on $[0,T_{n,\max})$.
So, we can proceed in a standard way and extend $\psi_{u_n}$ in $T_{n,\max}$, and in a neighborhood of $T_{n,\max}$, by the assumption of local well-posedness.
This is in contradiction with the maximality of $T_{n,\max}$, and then \eq{1821} follows.

Now, by Lemma \ref{Lcomp} we can find $u_0\in\mathcal{M}^r_c$ such that $u_n$ converges to $u_0$,  strongly in $\Sigma_{A,V}$.  
By the continuity of $\|\psi_{u_n}(t)\|^2_{\dot{\Sigma}_{A,V}}$, and since $\dist (\mathcal{M}^r_c,\partial B(r))_{\Sigma_{A,V}}>0$ by \eq{ntang}, from \eqref{1821} we infer that there exists $t_n\in (0,T_{n,\max})$ such that
\beq
\label{1847}
\|\psi_{u_n}(t_n)\|^2_{\dot{\Sigma}_{A,V}}=r.
\eeq
By mass and energy preservation, as $n\to\infty$ we obtain
\begin{equation*}
\|\psi_{u_n}(t_n,\cdot)\|^2_2 = \|u_n(\cdot)\|^2_2\to \|u_0\|^2_2 = c\ \mbox{ and }\  \mathcal{E}(\psi_{u_n}(t_n,\cdot)) = \mathcal{E}(u_n)\to \mathcal{E}(u_{0}) = m^r_c.
\end{equation*}
Arguing as in Lemma \ref{Lcomp}, we see that there exists a $u_*\in \mathcal{M}^r_c$ satisfying $\psi_{u_n}(t_n,\cdot)\to u_*$ in $\Sigma_{A,V}$, up to a subsequence. 
In particular, taking into account \eqref{1847}, we infer $\|u_*\|^2_{\dot{\Sigma}_{A,V}} = r$ and $\mathcal{E}(u_*) = m^r_c$, contrary to \eq{ntang}. 
This completes the proof of the first claim.

\medskip

To end the proof of the stability, we argue by contradiction and suppose that $\mathcal{M}^r_c$ is not orbitally stable. 
Then, taking into account Lemma \ref{Lcomp}, we find $u_0\in\mathcal{M}^r_c$,   $\va_0>0$,  a sequence of initial data $u_{n}\in \Sigma_{A,V}$ satisfying
\begin{equation}\label{eq3.25}
\lim\limits_{n\to\infty}\|u_n-u_0\|_{\Sigma_{A,V}} = 0,
\end{equation}
a sequence of time $t_n\in [0,\infty)$  and a sequence of flows $\psi_{u_n}(t,\cdot)$ with $\psi_{u_n}(0,\cdot) = u_n$,  $n\in \N$, such that
\begin{equation}\label{ueq3.26}
\inf_{u\in\mathcal{M}^r_c}\|\psi_{u_n}(t_n,\cdot)-u\|_{\dot{\Sigma}_{A,V}}\ge\va_0.
\end{equation}
By \eq{eq3.25}, Propositions \ref{pr2.5},  and \eqref{ntang},
it holds
\begin{equation}
\label{1255}
\lim_{n\to\infty}\|u_n\|_{\dot{\Sigma}_{A,V}}=\|u_{0}\|_{\dot{\Sigma}_{A,V}}< r.
\end{equation}
%Hence, by  \eq{eq3.25} and Lemma \ref{le3.5}, there exists a $\hat{c}^* = \hat{c}^*(r)>0$ such that the flows $\psi_{u_n}(t,\cdot)$  exist globally in time for all $0<c<\hat{c}^*$.
By preservation of mass and energy, and by \eq{1255}, we infer
\begin{equation}\label{eq3.28}
\|\psi_{u_n}(t,\cdot)\|^2_2 = \|u_n\|^2_2\to \|u_0\|^2_2 = c\ \mbox{ and }\  \ \mathcal{E}(\psi_{u_n}(t,\cdot)) = \mathcal{E}(u_n)\to \mathcal{E}(u_{0}) = m^r_c,\quad \forall t\ge 0.
\end{equation}

 We claim that \[\sharp\{n\in\N:\|\psi_{u_n}(t_n,\cdot)\|^2_{\dot{\Sigma}_{A,V}}> r\}<\infty.\]
If not,
$$
\|\psi_{u_{n}}(t_{n},\cdot)\|_{\dot{\Sigma}_{A,V}}> r\ \ \forall n\in\N,
$$
and $ \|\psi_{u_{n}}( 0,\cdot)\|_{\dot{\Sigma}_{A,V}}<r$ by \eq{1255}, up to a subsequence.
By continuity, for every $n\in\N$, there exists $t^*_{n}\in[0,t_{n})$ such that
$$
\|\psi_{u_{n_k}}(t^*_{n},\cdot)\|_{\dot{\Sigma}_{A,V}}=r.
$$
But by \eq{eq3.28} we can argue as in Lemma \ref{Lcomp}, and deduce the existence of $\tilde{u}_*\in \Sigma_{A,V}$ with $\|\tilde{u}_*\|^2_2 = c$ and $\mathcal{E}(\tilde{u}_*) = m^r_c$, such that $\psi_{u_{n_k}}(t^*_{n_k},\cdot)\to \tilde{u}_*$ strongly in $\Sigma_{A,V}$. 
Then $\tilde{u}_*\in {S_c}\cap \partial B(r)$ is a minimizer for $m^r_c$, contrary to \eq{ntang}.

Finally, by the claim above, we can assume that $(\psi_{u_n}(t_n,\cdot))_{n\in \N}\subset B(r)$. 
Arguing as above, we can prove that there exists a $u\in \mathcal{M}^r_c$ such that
$$
\lim\limits_{n\to\infty}\|\psi_{u_n}(t_n,\cdot) - u\|_{\dot{\Sigma}_{A,V}} = 0,
$$
contrary to \eq{ueq3.26}. 
This completes the proof.

\end{proof}

%%%%%%%%%%%%%%%%%%%%%%%%%%%%%%%%%%%%%%%%%%%%%%%%%%%%%%%%%%%%%%%%%%%%%%%%%%%%%%

\section{Uniqueness and symmetry properties of the ground state}
%\label{s4}

%%%%%%%%%%%%%%%%%%%%%%%%%%%%%%%%%%%%%%%%%%%%%%%%%%%%%%%%%%%%%%%%%%%%%%%%%%%%%%

In this section, we give the proof of Theorem \ref{th1.3} and Corollary \ref{abth1.5}. 
Initially, we employ the implicit function theorem to demonstrate Theorem \ref{th1.3}.
Next, based on the same argument and the proof of Theorem \ref{th1.1}, we give an outline of the proof of Corollary \ref{abth1.5} after converting \eq{eq1.1} into a non-magnetic field coupling problem when $|\a|\le 1/2$ (see Lemma \ref{able5.1} below).

\medskip

%\subsection{Proof of Theorem \ref{th1.3}}\label{ss4.1}
Here, $u_{{c},A,V}$  is an energy ground state for \eq{eq1.1}, according to Part 2 of Theorem \ref{th1.1}.
To start, we transfer the mass restriction to the nonlinear term.
Notice that the function $v_{c,A,V} = u_{c,A,V}/\sqrt{c}$ is an energy ground state for the problem
\begin{equation*}
\left\{
  \begin{array}{ll}
    (i\nabla + A(x))^2 u + V(x)u - c|u|^2u = \la_{c,A,V}u, & \text{in}\ \R^3 \\
    \|u\|^2_2 = 1. &
  \end{array}
\right.
\end{equation*}
Moreover, it achieves the following infimum
\begin{equation*}
%\label{geq4.1}
\breve{m}^r_c = \inf_{u\in S_1 \cap B(r/c)}\mathcal{E}_c(u),
\end{equation*}
where
$$
\mathcal{E}_c(u) = \frac{1}{2}\|u\|^2_{\dot{\Sigma}_{A,V}} - \frac{c}{4}\|u\|^4_4.
$$
Observe that $\breve{m}^r_c = m^r_c/c$.
For brevity, we denote by $\hat{\psi}_{1}$ the first eigenfunction and by $\hat{\la}_1$ the corresponding first eigenvalue of $\hat{L}_{A,V}$.

\textsf{Proof of part 1 of Theorem \ref{th1.3}: uniqueness.} 
Denote by $V_1 = \{\la\hat{\psi}_1\, :\, \la\in\C\}$ and $V_1^{\bot}$ the orthogonal space in $L^2(\R^3,\C)$ to $V_1$ with respect to the $L^2$ scalar product. 
Note that $V_1$ is the eigenspace related to $\hat \lambda_1$ because it is assumed to be simple and that $V_1^{\bot} =   \spant\{\hat{\psi}_{2,A,V},\hat{\psi}_{3,A,V},\ldots\}$ according to Lemma \ref{Ale2.8}. 
For each $\varphi\in V_1$ with $\|\varphi\|^2_2=1$, define the function $F^{\varphi}:\big[V_1^{\bot}\cap{\Sigma}_{A,V}\big] \times \R^3\to \Sigma^*_{A,V}$ as
\begin{align}\label{9999}
F^{\varphi}(u_1,u_2,\la,s,{c})
&= \big[(i\nabla + A(x))^2 + V(x)\big](u+(1+s)\varphi)\\
&\quad - {c}|u+(1+s)\varphi|^2(u+(1+s)\varphi) - (\hat{\la}_{1}+\la)(u+(1+s)\varphi), \nonumber
\end{align}
where $\Sigma^*_{A,V}$ is the dual space of ${\Sigma}_{A,V}$ and $u=u_1+iu_2\in V_1^{\bot}\cap{{\Sigma}_{A,V}}$ ($u_1$ and $u_2$ are the real and imaginary parts of $u$, respectively). 
We have
\begin{align*}
&(i)\ F^{\varphi}(0,0,0,0,0) = 0,\\
&(ii)\ F^{\varphi}_{u_1}(0,0,0,0,0) = \hat{L}_{A,V}  - \hat{\la}_{1},\\
&(iii)\ -i F^{\varphi}_{u_2}(0,0,0,0,0) = \hat{L}_{A,V} - \hat{\la}_{1},\\
&(iv)\ F^{\varphi}_\la (0,0,0,0,0)= -\varphi,\\
&(v)\ F^{\varphi}_s(0,0,0,0,0)=0.
\end{align*}
Furthermore, for every
\[(w,\hat{\la}) = (w_1+iw_2,\hat{\la})\in \big[V_1^{\bot}\cap{{\Sigma}_{A,V}}\big]\times \R,\]
it holds
\begin{align*}
&\quad(F^{\varphi}_{u_1}(0,0,0,0,0), F^{\varphi}_{u_2}(0,0,0,0,0), F^{\varphi}_\la (0,0,0,0,0))(w_1,w_2,\hat{\la})\\
& = (\hat{L}_{A,V}-\hat{\la}_{1})w_1 + i(\hat{L}_{A,V}-\hat{\la}_{1})w_2 - \hat{\la}\varphi = (\hat{L}_{A,V}-\hat{\la}_{1})w - \hat{\la}\varphi.
\end{align*}
Hence, by the simplicity of $\hat{\la}_1$, the map
\[F^{\varphi}_{u_1,u_2,\la}(0,0,0,0,0):\big[V_1^{\bot}\cap{\Sigma}_{A,V}\big]\times \R\to \Sigma^*_{A,V}\]
is an isomorphism. 
Therefore, by the implicit function theorem (\cite[Theorems 1.2.1 and 1.2.3]{KC-Zhang-2006}), there exist $\de_{1,\varphi}>0$, $\de_{2,\varphi}>0$ and a unique function  
\[(
u_1(s,{c}),u_2(s,{c}),\la(s,{c}))\in C^1(B_{\de_{1,\varphi}}(0,0);B_{\de_{2,\varphi}}(0,0))
\]
such that
\begin{equation}\label{uueq4.2}
\left\{
  \begin{array}{ll}
    F^{\varphi}(u_1(s,{c}),u_2(s,{c}),\la(s,{c}),s,{c}) = 0, & \\
    u_1(0,0) = 0,u_2(0,0)=0,\ \la(0,0) = 0&
  \end{array}
\right.
\end{equation}
and
\begin{equation}\label{ueq4.12}
\left[F^{\varphi}_{u_1}\frac{\partial u_1}{\partial s}(0,0) +F^{\varphi}_{u_2}\frac{\partial u_2}{\partial s}(0,0) + F^{\varphi}_{\la}\frac{\partial \la}{\partial s}(0,0)\right]_{(0,0,0,0,0)} = 0.
\end{equation}
Note that \eq{ueq4.12} is equivalent to
\begin{equation*}
%\label{ueq4.13}
(\hat L_{A,V}-\hat{\la}_{1})\frac{\partial}{\partial s} u(0,0) - \left(\frac{\partial \la}{\partial s}(0,0)\right)\varphi = 0
\end{equation*}
by the computation above. 
Then
\begin{equation}\label{uueq4.14}
\frac{\partial}{\partial s} u(0,0)=0,\quad \frac{\partial \la}{\partial s}(0,0)=0
\end{equation}
by the simplicity of $\hat{\la}_1$.

Now let
$$
\hat{u}(s,{c}) = (1+s)\varphi+u(s,{c}),\quad (s,{c})\in B_{\de_1}(0,0),
$$
and define 
$$
f(s,{c}) = |\hat{u}(s,{c})|^2_2 = (1+s)^2 + \int_{\R^3}|u(s,{c})|^2,\quad (s,{c})\in B_{\de_1}(0,0).
$$
By \eq{uueq4.2} and \eq{uueq4.14}, we have
$$
f(0,0)=1\ \text{and}\ f_s(0,0) = 2 + 2\Re\int_{\R^3}u_s(0,0)\overline u(0,0) = 2.
$$
Then, by applying the implicit function theorem again, there exist $\de_{\varphi} \in(0,\de_{1,\varphi})$ and a unique $s=s({c})\in C^1(B_{\de_{\varphi}}(0);B_{ \de_{1,\varphi}}(0))$ such that
$$
f(s({c}),{c}) = f(0,0) = 1,\quad {c}\in B_{\de_{\varphi}}(0).
$$
This and \eq{uueq4.2} show that, for ${c}\in B_{\de_{\varphi}}(0)$, there exists a unique function:
$$
(u({c}):=u(s({c}),{c}),\la({c}):= \hat{\la}_{0}+\la(s({c}),{c}))\in C^1(B_{\de_{\bl \varphi}}(0);B_{\de_{2,\varphi}}(\varphi,\hat{\la}_{0}))
$$
such that
\begin{equation}\label{ueq4.13}
  \left\{
    \begin{array}{ll}
      \la(0) = \hat{\la}_{1},\ u(0)=\varphi, &   \\
      (i\nabla + A(x))^2 u({c}) + V(x)u({c}) - {c} |u({c})|^2u({c}) -\la({c})u({c}) = 0, & \\
      |u({c})|^2_2=1.&
    \end{array}
  \right.
\end{equation}

Now suppose, contrary to our statement, that for every $c>0$ there exist two minimizers $v^j_{c,A,V}$, $j=1,2$, for $\breve{m}^r_c$, which are different up to a phase factor, i.e. there exists no $\theta\in\R$ such that $v^1_{c,A,V} = e^{i\theta} v^2_{c,A,V}$. 
Taking into account that $\hat\psi_1$ is unique, it can be easily checked that there exist two constants $\theta_j\in\R$, $j=1,2$, such that
$$
v^j_{c,A,V}\to e^{-i\theta_j}\hat{\psi}_1\quad \text{strongly in}\ \Sigma_{A,V},\ j=1,2,
$$
as $c\to 0^+$, up to subsequences. 
Consequently,
$$
e^{i\theta_j}v^j_{c,A,V}\to \hat{\psi}_1\quad \text{strongly in}\ \Sigma_{A,V},\ j=1,2,
$$
as $c\to 0^+$.  
Then when $|c|>0$ is small, by \eq{ueq4.13}, we have
\[
e^{i\theta_j}v^{j}_{c,A,V}\in B_{\de_{\hat{\psi}_1}},\  j=1,2, \  \text{ and }\  e^{i\theta_1}v^1_{c,A,V} = e^{i\theta_2}v^2_{c,A,V},
\]
which is a contradiction. 
This completes the proof of part 1.
\qed

\medskip

\textsf{Proof of part 2 of Theorem \ref{th1.3}: symmetry.}\quad 
Denote the cylindrical coordinates as $(x_1,x_2,$ $x_3)=(r\cos\theta,r\sin\theta,x_3)$, where $r=\sqrt{x^2_1+x^2_2}$.
 Recall that the magnetic field $A$ satisfies $A(gx) = g(A(x))$ for all $g\in G$, $x\in\R^3$, and notice that $(x_1,x_2,$ $x_3)=g^+_\theta(r,0,x_3)=g^-_\theta(r,0,-x_3)$, $\forall\theta \in [0,2\pi)$,  where $g^\pm_{\theta}$ are defined in \eq{ueq1.13}.
Then, we must have 
$$
A(x_1,x_2,x_3) = g^+_{\theta}\left(
                \begin{array}{c}
                  A_1(r,0,x_3) \\
                  A_2(r,0,x_3) \\
                  A_3(r,0,x_3) \\
                \end{array}
              \right)
              =
              g^-_{\theta}\left(
                \begin{array}{c}
                  A_1(r,0,-x_3) \\
                  A_2(r,0,-x_3) \\
                  A_3(r,0,-x_3) \\
                \end{array}
              \right),
$$
that is $A_3$ is odd and $A_1, A_2$ are even, in the $x_3$ variable. 
As a consequence, since $A$ is assumed to be linear and skew-symmetric, a simple calculation shows that $A_1(r,0,x_3)= A_3(r,0,x_3)=0$ and $A_2(r,0,|x_3|)= h r$, for a constant $h\in\R$, 
that is $A$ has the form
\begin{equation}
\label{ueq1.12}
 \left(
      \begin{array}{ccc}
        0 & -h & 0 \\
        h & 0 & 0\\
        0 & 0 & 0 \\
      \end{array}
    \right)
\end{equation}
%Since the potential $V$ is assumed to satisfy $V(gx) = V(x)$ for all $g\in G$ and $x\in\R^3$, $V$ is cylindrically symmetric with respect to the two variables $x_1$ and $x_2$, i.e.,  $V = V(r,x_3)$.
(see also \cite{Bonheure-Cingolani-Nys-CV-2016}).

For every $g\in G$, it can be checked by the assumption on $V$ and $A$ that $\mathcal{E}(v_{c,A,V}(g\, \cdot)) = \mathcal{E}(v_{c,A,V})$. 
Hence $v_{c,A,V}(g\, \cdot)$  also achieves \[\breve{m}^r_c = \inf_{u\in S_1 \cap B(r/c)}\mathcal{E}_c(u),\]
i.e., $v_{c,A,V}(g\, \cdot)$ is also a minimizer of $\breve{m}^r_c$.
By the uniqueness in part 1, for every $g\in G$, there exists a constant $\theta_g \in [0,2\pi)$ such that
\begin{equation*}
%\label{ueq4.4}
v_{c,A,V} = e^{i\theta_g}v_{c,A,V}(g\, \cdot).
\end{equation*}
By the exponential decay of the solutions (Remark \ref{Rreg}), we can compute
\begin{align*}
\int_{\R^3}x|v_{c,A,V}(x)|^2 = \int_{\R^3}x|v_{c,A,V}(gx)|^2 = g^{-1}\int_{\R^3}x|v_{c,A,V}(x)|^2 dx
\end{align*}
for all $g\in G$. Consequently, it holds
$$
\int_{\R^3}x|v_{c,A,V}(x)|^2 = 0
$$
and then
\begin{equation}\label{eq5-6}
\Re(i(A(x)\cdot \nabla v_{c,A,V})\, {\bar v_{c,A,V} }) = 0
\end{equation}
by \cite[Proposition 5.1, Lemma 5.2]{Bonheure-Nys-Schaftingen-JMPA-2019}. 

\medskip 

 {\em Claim:} \begin{align}\label{eq5-6-1}
\begin{split}
 m^r_c &= \inf_{{u\in\Sigma_{A,V},\|u\|^2_2 = c, \|u\|^2_{\dot{\Sigma}_{A,V}}\le r}}\mathcal{E}(u)
= \inf_{{u\in{\Sigma}_{A,V},\|u\|^2_2 = c,\|u\|^2_{\dot{\Sigma}_{A,V}}\le r}}J(u)
\\
&
= \inf_{{u\in\hat{\Sigma}_{A,V},\|u\|^2_2 = c,\|u\|^2_{\ddot{\Sigma}_{A,V}}\le r}}J(u)
= \inf_{{u\in\hat{\Sigma}_{A,V},\|u\|^2_2 = c, \|u\|^2_{\ddot{\Sigma}_{A,V}}\le r}}J(|u|)\\
&
= \inf_{{u\in\hat{\Sigma}^\Re_{A,V},\|u\|^2_2 = c,\|u\|^2_{\ddot{\Sigma}_{A,V}}\le r}}J(|u|)
= \inf_{{u\in\hat{\Sigma}^\Re_{A,V},\|u\|^2_2  = c,
\|u\|^2_{\ddot{\Sigma}_{A,V}}\le r}}J(u),
\end{split}
\end{align}
(see Notations \ref{SNot}),
where 
\begin{equation}\label{eq5-7-3}
J(u) := \frac{1}{2}\Big(\int_{\R^3}|\nabla u|^2 + \int_{\R^3}\big(|A(x)|^2 + V(r,x_3)\big)|u|^2\Big) - \frac{1}{4}\int_{\R^3}|u|^4.
\end{equation}
%For a function $u$ in ${\Sigma}_{A,V}$, we write its Fourier decomposition in cylindrical coordinates as
%$$
%u(r,\theta,x_3)=\sum_{n\in \Z}c_n(r,x_3)e^{in\theta}\ \ \forall r,x_3>0,\ \theta\in[0,2\pi),
%$$
%where
%\begin{equation*}
%%\label{ueq4.9}
%c_n(r,x_3) = \frac{1}{2\pi}\int^{2\pi}_0u(r,\theta,x_3)e^{-in\theta}d\theta\ \ \forall n\in \Z.
%\end{equation*}
%According to the proof of part 2 in Proposition \ref{leA.1}\todogr{Mettere qui}, we have
%\begin{align*}
%\begin{split}
%&\quad\int_{\R^3}\left|\left(i\frac{\partial_\theta}{r} + \frac{\a}{r}\right)u\right|^2dx\\
%& = \int^{\infty}_{-\infty}dx_3\int^{\infty}_0dr\int^{2\pi}_0\sum_{n\in\Z}|c_n(r,z)|^2\Big(\frac{-n}{r}+\frac{\a}{r}\Big)^2rd{\theta}\\
%&\ge \int^{\infty}_{-\infty}dx_3\int^{\infty}_0dr\int^{2\pi}_0\sum_{n\in Z}\left|c_n(r,x_3)\frac{\a}{r}\right|^2rd\theta\\
%& = \int_{\R^3}|\a(-x_2/r^2,x_1/r^2,0)u|^2dx
%\end{split}
%\end{align*}
%if $|\a|\le\frac{1}{2}$.
%Hence
%\begin{align*}
%%\label{eqA.11}
%\begin{split}
%&\quad\int_{\R^3}\big(|(i\nabla + A(x))u|^2 + V(x)|u|^2\big)dx\\
%&\ge \int_{\R^3}|\nabla_{r,x_3} u|^2dx + \int_{\R^3}\big(|\a(-x_2/r^2,x_1/r^2,0)|^2 + V(r,x_3)\big)|u|^2dx
%\end{split}
%\end{align*}

\medskip

In fact, observe first that, by \eq{eq5-6}, it holds
$$
\mathcal{E}(u_{c,A,V}) = J(u_{c,A,V}).
$$
Then, combing with the diamagnetic inequality in Lemma \ref{lle2.2}, we have
\begin{align}\label{abeq5.4}
\begin{split}
&\quad\inf_{u\in\Sigma_{A,V},\|u\|^2_2 = c,\|u\|^2_{\dot{\Sigma}_{A,V}}\le r}\mathcal{E}(u)\ge \inf_{u\in{\Sigma}_{A,V},\|u\|^2_2 = c,\|u\|^2_{\dot{\Sigma}_{A,V}}\le r}J(u)\\
&\ge \inf_{u\in\hat{\Sigma}_{A,V},\|u\|^2_2 = c,\|u\|^2_{\ddot{\Sigma}_{A,V}}\le r}J(u)\ge \inf_{u\in\hat{\Sigma}_{A,V},\|u\|^2_2 = c,\|u\|^2_{\ddot{\Sigma}_{A,V}}\le r}J(|u|)\\
&\ge \inf_{u\in\hat{\Sigma}^\Re_{A,V},\|u\|^2_2 = c,\|u\|^2_{\ddot{\Sigma}_{A,V}}\le r}J(|u|) =  \inf_{u\in\hat{\Sigma}^\Re_{A,V},\|u\|^2_2 = c,\|u\|^2_{\ddot{\Sigma}_{A,V}}\le r}J(u);
\end{split}
\end{align}
here, the last equality holds because
$$
J(u) = J(|u|)\ \text{and}\ \||u|\|^2_{\ddot{\Sigma}_{A,V}}=  \|u\|^2_{\ddot{\Sigma}_{A,V}},\qquad \forall u\in\hat{\Sigma}^\Re_{A,V}. 
$$

Finally, $J(|u|) = \mathcal{E}(|u|)$, for all $u\in\hat{\Sigma}^\Re_{A,V}$, so that 
\begin{equation}
\label{eq5-6-2}
\inf_{u\in\hat{\Sigma}^\Re_{A,V},\|u\|^2_2 = c,\|u\|^2_{\ddot{\Sigma}_{A,V}}\le r}J(u) \ge \inf_{u\in\Sigma_{A,V},\|u\|^2_2 = c,\|u\|^2_{\dot{\Sigma}_{A,V}}\le r}\mathcal{E}(u)
\end{equation}
and the identities in \eqref{eq5-6-1} hold. 
This proves the claim.

\medskip
%Let $(u_{c,A,V},\la_{c,A,V})$ be a solution obtained by Theorem \ref{th1.1}. 
By the claim, $|u_{c,A,V}|$ achieves the infimum
\begin{equation}\label{abeq4.11}
\inf_{u\in\hat{\Sigma}^\Re_{A,V},\|u\|^2_2 = c,\|u\|^2_{\ddot{\Sigma}_{A,V}}\le r}J(u),
\end{equation}
and it must be true that $|\nabla u_{c,A,V}| = |\nabla |u_{c,A,V}||$.
Hence, if we write $u_{c,A,V}=:u=u_1 + iu_2$, with $u_1,u_2$ real-valued functions, then a straight computation provides $u_1\nabla u_2 = u_2\nabla u_1$. 
Then $\nabla u = \sgn({\bar u}) \nabla|u|$ a.e. in $\R^3$ and
\begin{equation}\label{abeq4.16}
{ \nabla \sgn(\bar u) = \frac{\nabla u-\sgn{\bar u}\nabla |u|}{|u|} = 0}.
\end{equation}
Consequently, $\sgn({ \bar u})\equiv z\in\C$ with $|z| = 1$ a.e. in $\R^3$. 
This means that $u_{c,A,V}$ is nonnegative up to a constant phase, i.e., $u_{c,A,V} = |u_{c,A,V}|e^{i\theta}$ for some $\theta\in\R$. 

\smallskip

Now, one can proceed in a standard way by Schwartz symmetrization (see for example Bellazzini et al. in \cite{Bellazzni-et.al.-CMP-2017}) and conclude that $|u_{c,A, V}|$ is cylindrically symmetric around $(x_1,x_2)$. 

\smallskip

Moreover, $|u_{c,A,V}|$ satisfies the Schr\"odinger equation 
\begin{equation*}
%\label{eq5-7-1}
\left\{
  \begin{array}{ll}
    -\Delta u + (V(x) + |A|^2-\la)u = |u|^2u, &\ \text{in}\ \R^3  \\
    \ds\int_{\R^3}|u|^2=c &
  \end{array}
\right.
\end{equation*}
with a $C^1(\R^3\setminus\{0\})$ potential. 
Then, by standard regularity results $|u_{c,A,V}|\in C^{2}(\R^3\backslash\{0\})$, is a classical solution (see also Remark \ref{Rreg}). 
Applying Harnack's inequality, we conclude that it is strictly positive on $\R^3\backslash\{0\}$.
\qed

\bigskip

Before proving Corollary \ref{abth1.5}, we convert the minimizing problem considered before into a new problem without the magnetic coupling effect.

\begin{lemma}\label{able5.1}
Let $|{\hat{\a}}|\le \frac{1}{2}$. Then there hold
\begin{align}\label{eq5-9-1}
\begin{split}
&\quad\inf_{{u\in\Sigma_{\mathcal{A}_{\hat{\a}},V},\|u\|^2_2 = c, \|u\|^2_{\dot{\Sigma}_{\mathcal{A}_{\hat{\a}},V}}\le r}}\mathcal{E}(u)
= \inf_{{u\in{\Sigma}_{\mathcal{A}_{\hat{\a}},V},\|u\|^2_2 = c,\|u\|^2_{\dot{\Sigma}_{\mathcal{A}_{\hat{\a}},V}}\le r}}J(u)
\\
&
= \inf_{{u\in\hat{\Sigma}_{\mathcal{A}_{\hat{\a}},V},\|u\|^2_2 = c,\|u\|^2_{\ddot{\Sigma}_{\mathcal{A}_{\hat{\a}},V}}\le r}}J(u)
= \inf_{{u\in\hat{\Sigma}_{\mathcal{A}_{\hat{\a}},V},\|u\|^2_2 = c, \|u\|^2_{\ddot{\Sigma}_{\mathcal{A}_{\hat{\a}},V}}\le r}}J(|u|)\\
&
= \inf_{{u\in\hat{\Sigma}^\Re_{\mathcal{A}_{\hat{\a}},V},\|u\|^2_2 = c,\|u\|^2_{\ddot{\Sigma}_{\mathcal{A}_{\hat{\a}},V}}\le r}}J(|u|)
= \inf_{{u\in\hat{\Sigma}^\Re_{\mathcal{A}_{\hat{\a}},V},\|u\|^2_2  = c,
\|u\|^2_{\ddot{\Sigma}_{\mathcal{A}_{\hat{\a}},V}}\le r}}J(u),
\end{split}
\end{align}
where 
$
J(u)
$ is defined as in \eq{eq5-7-3}.
\end{lemma}

\begin{proof}
 We proceed as for \eq{eq5-6-1}.
In particular, to confirm the validity of equations \eq{abeq5.4}--\eq{eq5-6-2}, we have only to prove that the special form of $\mathcal{A}_{\hat{\a}}$ leads to:
\begin{align}
\label{908}
\begin{split}
&\inf_{u\in\Sigma_{\mathcal{A}_{\hat{\a}},V},\|u\|^2_2 = c,\|u\|^2_{\dot{\Sigma}_{\mathcal{A}_{\hat{\a}},V}}\le r}\mathcal{E}(u)\ge \inf_{u\in{\Sigma}_{\mathcal{A}_{\hat{\a}},V},\|u\|^2_2 = c,\|u\|^2_{\dot{\Sigma}_{\mathcal{A}_{\hat{\a}},V}}\le r}J(u).
%\\
%&\ge \inf_{u\in\hat{\Sigma}_{\mathcal{A}_{\hat{\a}},V},\|u\|^2_2 = c,\|u\|^2_{\ddot{\Sigma}_{\mathcal{A}_{\hat{\a}},V}}\le r}J(u)\ge \inf_{u\in\hat{\Sigma}_{\mathcal{A}_{\hat{\a}},V},\|u\|^2_2 = c,\|u\|^2_{\ddot{\Sigma}_{\mathcal{A}_{\hat{\a}},V}}\le r}J(|u|)\\
%&\ge \inf_{u\in\hat{\Sigma}^\Re_{\mathcal{A}_{\hat{\a}},V},\|u\|^2_2 = c,\|u\|^2_{\ddot{\Sigma}_{\mathcal{A}_{\hat{\a}},V}}\le r}J(|u|){\bl = } \inf_{u\in\hat{\Sigma}^\Re_{\mathcal{A}_{\hat{\a}},V},\|u\|^2_2 = c,\|u\|^2_{\ddot{\Sigma}_{\mathcal{A}_{\hat{\a}},V}}\le r}J(u),
\end{split}
\end{align}
that is  the first inequality of \eq{abeq5.4}.

For a function $u$ in ${\Sigma}_{\mathcal{A}_{\hat{\a}},V}$, we write its Fourier decomposition in cylindrical coordinates as
\beq\label{1150}
u(r,\theta,x_3)=\sum_{n\in \Z}c_n(r,x_3)e^{in\theta}\ \ \forall r,x_3>0,\ \theta\in[0,2\pi),
\eeq
where
\begin{equation*}
%\label{ueq4.9}
c_n(r,x_3) = \frac{1}{2\pi}\int^{2\pi}_0u(r,\theta,x_3)e^{-in\theta}d\theta\ \ \forall n\in \Z.
\end{equation*}

Note that
\begin{align*}
\nabla u &= \left(\frac{x_1}{r}\partial_ru,\frac{x_2}{r}\partial_ru,\partial_{x_3}u\right) + \big(\partial_{\theta}u \partial_{x_1}\theta , \partial_{\theta}u \partial_{x_2}\theta,0\big).
\end{align*}
Hence
\begin{align*}
\begin{split}
(\mathcal{A}_{\hat{\a}}(x)\cdot \nabla u)
&= {\hat{\a}}\left(\frac{-x_2}{|x|^2},\frac{x_1}{|x|^2},0\right)\cdot \nabla u\\
&= {\hat{\a}}\left(\frac{-x_2}{|x|^2},\frac{x_1}{|x|^2},0\right)\cdot \big(\partial_{\theta}u \partial_{x_1}\theta, \partial_{\theta}u \partial_{x_2}\theta,0\big)\\
& = \frac{{\hat{\a}}\partial_{\theta} u}{|x|^2}
\end{split}
\end{align*}
and
\begin{align*}
\begin{split}
|(i\nabla + \mathcal{A}_{\hat{\a}}(x))u|^2
& = |\nabla u|^2 + |\mathcal{A}_{\hat{\a}}(x)|^2|u|^2 + 2\Re i(\mathcal{A}_{\hat{\a}}(x)\cdot \nabla u )\, \bar{u}\\
& = |\nabla u|^2 + |\mathcal{A}_{\hat{\a}}(x)|^2|u|^2 + 2\Re i\frac{{\hat{\a}}\partial_{\theta} u}{|x|^2}\, \bar{u}.
%|\nabla u|^2 = |\nabla_{r,x_3}u|^2 + \frac{| \partial_{\theta}  u|^2}{r^2}
\end{split}
\end{align*}

Then, we obtain
\begin{align}\label{eqA.9bis}
\int_{\R^3}|(i\nabla + \mathcal{A}_{\hat{\a}}(x))u|^2dx
&= \int_{\R^3}|{ \nabla_{r,x_3}}  u|^2dx + \int_{\R^3}\Big(\frac{|\hat{\a}|^2r^2|u|^2}{|x|^4} + \frac{1}{r^2}|\partial_\theta u|^2+\frac{2\hat{\a}}{|x|^2}\Re i{\bar{u}\partial_\theta u}\Big) \nonumber \\
&= \int_{\R^3}| \nabla_{r,x_3}  u|^2dx + \int_{\R^3}\Big|\Big(\frac{\hat{\a}r}{|x|^2} + i\frac{\partial_\theta}{r}\Big)u\Big|^2dx
\end{align}
By the Fourier decomposition \eq{1150}, when $|\hat{\a}|\le \frac{1}{2}$, we have 
\begin{align}
\label{1243}
\begin{split}
&\quad\int_{\R^3}\left|\left(\frac{{\hat{\a}}r}{|x|^2}+ i\frac{\partial_\theta}{r} \right)u\right|^2dx\\
& = \int^{\infty}_{-\infty}dx_3\int^{\infty}_0dr\int^{2\pi}_0\sum_{n\in\Z}|c_n(r,x_3)|^2\Big(\frac{{\hat{\a}}r}{r^2 + x^2_3}-\frac{n}{r}\Big)^2r\,d{\theta}\\
&\ge \int^{\infty}_{-\infty}dx_3\int^{\infty}_0dr\int^{2\pi}_0\sum_{n\in \Z}\left|c_n(r,x_3)\frac{{\hat{\a}}r}{r^2 + x^2_3}\right|^2r\, d\theta\\
& = \int_{\R^3}\frac{|\hat{\a}|^2r^2}{(r^2+x^2_3)^2}{|u|^2}dx,
\end{split}
\end{align}
where we have used the fact that $\Big|\frac{\hat{\a}r}{r^2+x^2_3}-\frac{n}{r}\Big|\ge \Big|\frac{\hat{\a}r}{r^2+x^2_3}\Big|$ if $|\hat{\a}|\le \frac{1}{2}$.

By \eq{eqA.9bis} and \eq{1243} we obtain
\begin{align}
\label{P29}
%\label{eqA.11}
%\begin{split}
%&
\int_{\R^3}\big(|(i\nabla + \mathcal{A}_{\hat{\a}}(x))u|^2 + V(x)|u|^2\big)dx
%\\
%&
\ge \int_{\R^3}|{ \nabla_{r,x_3}} u|^2dx + \int_{\R^3}\Big(\frac{|\hat{\a}|^2r^2}{|x|^4} + V(r,x_3)\Big)|u|^2dx,
%\end{split}
\end{align}
 which implies \eq{908} and concludes the proof.

\end{proof}

 \begin{proof}[\textsf{Proof of Corollary \ref{abth1.5}}]   
 First note that the same minimizing procedure in Theorem \ref{th1.1} implies the infimum 
\begin{equation*}
%\label{eq5-9-2}
\inf_{{u\in\hat{\Sigma}^\Re_{\mathcal{A}_{\hat{\a}},V},\|u\|^2_2  = c,
\|u\|^2_{\ddot{\Sigma}_{\mathcal{A}_{\hat{\a}},V}}\le r}}J(u)
\end{equation*}
can be achieved and its minimizer is a solution to \eq{abeq1.5} when $0<c<c_{0,V_{\hat{\a}}}(r)$. Let $\tilde{u}_{c,0,V_{\hat{\a}}}$ be a corresponding minimizer. By the standard symmetrization, regularity procedure and Harnack's inequality mentioned after \eq{abeq4.16}, we know that $\tilde{u}_{c,0,V_{\hat{\a}}}$ belongs to $C^2(\R^3\backslash\{0\})$ and is strictly positive and cylindrically symmetric w.r.t. $(x_1,x_2)$. Especially, in turn, by Lemma \ref{able5.1}, $\tilde{u}_{c,0,V_{\hat{\a}}}$ also achieves the infimum \eq{abeq1.2}, which gives the existence of minimizer.

\smallskip

 On the other hand, if $u$ achieves \eq{abeq1.2}, then Lemma \ref{able5.1} implies that $|u|$ must achieves all the infimum exhibited in \eq{eq5-9-1}. Especially, by \eq{abeq4.16}, it holds 
$u = e^{i\theta}|u|$ for some $\theta\in \R$. The symmetry property of $u$ is also similar to that in Theorem \ref{th1.3}. 

\smallskip

 In Proposition \ref{P2} of the coming Section \ref{S5}, we see that the first eigenvalue of the linear part is simple.  Then, applying point 1 of Theorem \ref{th1.3}, we establish uniqueness.
%Hence, if we write $u_{c,\mathcal{A}_{\hat{\a}},V}=:u=u_1 + iu_2$, with $u_1,u_2$ real-valued functions, then a straight computation provides $u_1\nabla u_2 = u_2\nabla u_1$. 
%Then $\nabla u = \sgn({ \bar u}) \nabla|u|$ a.e. in $\R^N$ and
%\begin{equation}\label{abeq4.16}
%{ \nabla \sgn(\bar u) = \frac{\nabla u-\sgn{\bar u}\nabla |u|}{|u|} = 0}.
%\end{equation}
%Consequently, $\sgn({ \bar u})\equiv z\in\C$ with $|z| = 1$ a.e. in $\R^N$. 
%This means that $u_{c,\mathcal{A}_{\hat{\a}},V}$ is nonnegative up to a constant phase, i.e., $u_{c,\mathcal{A}_{\hat{\a}},V} = |u_{c,\mathcal{A}_{\hat{\a}},V}|e^{i\theta}$ for some $\theta\in\R$. 
%
%\smallskip
%
%Now, one can proceed in a standard way by Schwartz symmetrization (see for example Bellazzini et al. in \cite{Bellazzni-et.al.-CMP-2017}) and conclude that $|u_{c,\mathcal{A}_{\hat{\a}}, V}|$ is cylindrically symmetric about $(x_1,x_2)$. 
%
%\smallskip
%
%Moreover, $|u_{c,\mathcal{A}_{\hat{\a}},V}|$ satisfies the Schr\"odinger equation \eq{abeq1.5}, with a $C^1(\R^3\setminus\{0\})$ potential. 
%Then, by standard regularity results $|u_{c,\mathcal{A}_{\hat{\a}},V}|\in C^{2}(\R^3\backslash\{0\})$, is a classical solution (see also Remark \ref{Rreg}). 
%Applying Harnack's inequality,  we  conclude that it is strictly positive on $\R^3\backslash\{0\}$.
%
%\smallskip
%
%{ In Proposition \ref{P2} of the coming Section \ref{S5}, we see that the first eigenvalue of the linear part is simple.  Then, applying point 1 of Theorem \ref{th1.3}, we establish uniqueness.
%}
\end{proof}

%\deleted[id=R.M.]{Here I have deleted the old remark, as %you suggest.}

%\todobl[inline]{I know that this result is true if we can read the minimum problem for the magnetic case as a non magnetic one, as done in Lemma \ref{able5.1}. 
%Is this result true also in the other case? If yes, where is it written and proved?
%\\
%If we have not a correct reference, it is better to %write this remark with the auxiliary assumption that the %minimum problem is equivalent to a non-magnetic one, %that is the case we are studying. \\ {\vi  We both agree %with your opinion. Maybe we can delete this remark. }%{\bl I have deleted it.}}

%\vspace{0.5cm}

%%%%%%%%%%%%%%%%%%%%%%%%%%%%%%%%%%%%%%%%%%%%%%%%%%%%%%%%%%%%%%%%%%%%%%%%%%%%%%%%%%%%%%%%%%%%%%%%

\section{Simplicity of the first eigenvalue}
\label{S5}

%%%%%%%%%%%%%%%%%%%%%%%%%%%%%%%%%%%%%%%%%%%%%%%%%%%%%%%%%%%%%%%%%%%%%%%%%%%%%%%%%%%%%%%%%%%%%%%%

In this section, we present some concrete cases in which the first eigenvalue of the linear part is simple.
In detail, in Proposition \ref{leA.1} we show that the first eigenvalue of $\hat{L}_{A,V} = (i\nabla + A(x))^2 + V(x)$ is simple if $A$ is small, in Proposition \ref{P2} we analyze the magnetic potential $\mathcal{A}_\alpha$ and finally in Proposition \ref{crA.4} we consider the particular case introduced in Example \ref{ex1}.
As a consequence of the latter case, we can prove Theorem \ref{th1.6}.

\begin{proposition}\label{leA.1}
Let $A,V$ satisfy the assumptions \eq{V1} and \eq{V2}. 
For $t\in \R$,   let $\hat{L}_{tA,V}$ be the operator in \eq{1441} with magnetic potential $tA$.
Then there exists $t_0>0$ such that, for all $|t|<t_0$, the corresponding first eigenfunction $\hat{\psi}_{1,tA,V}$ is unique up to a constant phase.
\end{proposition}

\begin{proof}
The proof employs the implicit function theorem.
By Lemma \ref{LBar}, the magnetic operator $\hat L_{0,V}$ has the same eigenevalues of the corresponding real Schr\"odinger operator $L_{0,V}$, and the same eigenvectors, up to scalar factors in $\C$.
In particular, its first eigenvalue is simple, according to Lemma \ref{Ale2.7}.
Now, we decompose $L^2(\R^3,\C)$ into the direct sum of eigenspaces corresponding to the eigenvalues of the operator ${L}_{0,V}$, that is
$$
L^2(\R^3,\C) = \bigoplus_{k=1}^\infty \mathcal{H}_k,
$$
where $\mathcal{H}_k$ is the eigenspace corresponding to the $k$-th eigenvalue $\la_{k,0,V}$. 
Note that 
\[\mathcal{H}_1 = \{\tau\psi_{1,0,V}: \tau\in \C\} \text{ and } \mathcal{H}^{\bot}_1 = \bigoplus_{k=2}^\infty \mathcal{H}_k.\]
%Denote by $\{{\psi}_{0}\}^{\bot}$ the orthogonal space in $L^2(\R^3,\C)$ to ${\psi}_{0}$ with respect to the $L^2$ scalar product.
Denote $\la_1 = \la_{1,0,V}$ and define the function $F :\big[\mathcal{H}^{\bot}_1\cap \Sigma_{0,V}\big] \times \R^3\to \Sigma_{0,V}^*$ as
\begin{align*}
F(u_1,u_2,\la,t,s)
&= \big[(i\nabla + tA(x))^2 + V(x)\big](u+(1+s)\psi_{1,0,V}) - (\la_1+\la)(u+(1+s)\psi_{1,0,V})
\end{align*}
where $u=u_1+iu_2\in \mathcal{H}^{\bot}_1\cap{{\Sigma}_{0,V}}$ and  $\Sigma^*_{0,V}$ is the dual space of ${\Sigma}_{0,V}$. 
There hold
\beq
\label{1722}
\begin{array}{cl}
(i)& \   F(0,0,0,0,0) = ( \hat L_{0,V} - \la_1)\psi_{1,0,V} = 0,\\
(ii) &\ F_{u_1}(0,0,0,0,0) =  \hat L_{0,V}- \la_1 ,\\
(iii)&\   F_{u_2}(0,0,0,0,0) = i( \hat L_{0,V} - \la_1),\\
(iv)&\ F_\la (0,0,0,0,0)= -\psi_{1,0,V},\\
(v)&\ F_s(0,0,0,0,0)=0.
\end{array}
\eeq
For every
\[(w,\hat{\la}) = (u_1+iu_2,\hat{\la})\in \big[\mathcal{H}^{\bot}_1\cap{{\Sigma}_{0,V}}\big]\times \R,\]
we have
\begin{align*}
&\hspace{-3mm}(F_{u_1}(0,0,0,0,0), F _{u_2}(0,0,0,0,0), F _\la (0,0,0,0,0))(u_1,u_2,\hat{\la})\\
 &=  ( \hat L_{0,V} -\la_1)u_1 + i( \hat L_{0,V} -\la_1)u_2 - \hat{\la}\psi_{1,0,V} = ( \hat L_{0,V} -\la_1)w - \hat{\la}\psi_{1,0,V}.
\end{align*}
Then, it is readily seen that the operator
\begin{align*}
F_{u_1,u_2,\la}(0,0,0,0,0)&:\big[\mathcal{H}^{\bot}_1\cap \Sigma_{0,V}\big]\times \R\to \Sigma_{0,V}^*\\
(w,\hat{\la})&\mapsto (L_0-\la_1)w - \hat{\la}\psi_{1,0,V}
\end{align*}
 is an isomorphism. 
 Therefore, by the implicit function theorem (\cite[Theorems 1.2.1 and 1.2.3]{KC-Zhang-2006}) and \eq{1722}, there exist $\de_{1},\de_{2}>0$ and a unique function \[(u_1(t,s),u_2(t,s),\la(t,s))\in C^1(B_{\de_{1}}(0,0);B_{\de_{2}}(0,0))\]
 such that
\begin{equation}\label{eq4.12}
\left\{
  \begin{array}{ll}
    H(t,s):=F(u_1(t,s),u_2(t,s),\la(t,s),t,s) = 0,\quad \forall (t,s)\in  B_{\de_{1}}(0,0),& \\
    u_1(0,0) = 0,u_2(0,0)=0,\ \la(0,0) = 0, & \\
    \left[\frac{\partial H}{\partial s}\right]_{(t,s)=(0,0)}=(\hat L_{0,V}-\la_1)\frac{\partial u}{\partial s}(0,0) - \frac{\partial\la}{\partial s}(0,0)\psi_{1,0,V} = 0. &
  \end{array}
\right.
\end{equation}

Now let
$$
\hat{u}(t,s) = u(t,s)+(1+s)\psi_{1,0,V},\quad (t,s)\in B_{\de_{1}}(0,0),
$$
and define
$$
G(t,s) = |\hat{u}(t,s)|^2_2 = (1+s)^2 + \int_{\R^3}|u(t,s)|^2,\quad  (t,s)\in B_{\de_{1}}(0,0).
$$
By \eq{eq4.12}, we have
$$
G(0,0)=1\ \text{and}\ G_s(0,0) = 2 + 2\Re\int_{\R^3}u_s(0,0)\cdot\overline u(0,0) = 2.
$$
Then,  applying implicit function theorem again, there exist $\bar \de\in(0,\de_{1})$ and a unique $s=s(t)\in C^1(B_{\bar\de}(0);B_{\de_{1}}(0))$ such that
$$
G (t,s(t)) = G (0,0) = 1,\quad\forall t \in B_{\bar\de}(0).
$$
This and \eq{eq4.12} show that, for $t\in B_{\bar \de}(0)$, there exists a unique function 
$$
(u(t):=u(t,s(t)),\la(t):= \la_1 + \la(t,s(t)))\in C^1(B_{\bar\de}(0);B_{\de_{2}}(\psi_{1,0,V},\la_1))
$$
such that
\begin{equation}\label{eq4.13}
  \left\{
    \begin{array}{ll}
      \la(0) = \la_1,\ u(0)=\psi_{1,0,V}, &   \\
      (i\nabla + tA(x))^2 u(t) + V(x)u(t) - \la(t)u(t) = 0, & \\
      |u(t)|^2_2=1.&
    \end{array}
  \right.
\end{equation}

\smallskip

{\em Claim:}  according to \eq{eq4.13}, it remains to show that
\begin{equation*}
(e^{i\a }\hat{\psi}_{1,tA,V},\hat{\la}_{1,tA,V})\in B_{\bar \de}(\psi_{1,0,V},\la_1)\quad  (\text{in the topology of}\ \Sigma_{0,V}\times\R)
\end{equation*}
for some $\a\in\R$ when $|t|>0$ is small. 

\smallskip

By using $\psi_{1,0,V}$ in the variational characterization of the first eigenvalue, a direct computation provides
$$
\hat{\la}_{1,tA,V}\le\int_{\R^3}\big(|(i\nabla + t A(x))\psi_{1,0,V}|^2 + V(x)|\psi_{1,0,V}|^2\big)= \la_{1,0,V}+o_t(1).
$$
Hence, for small $\delta$ we see that $\{\hat{\psi}_{1,tA,V}\}_{|t|<\delta}$ is bounded in $\Sigma_{0,V}$, because 
$$
\hat{\la}_{1,tA,V}=\int_{\R^3}\big(|(i\nabla + t A(x))\hat{\psi}_{1,tA,V}|^2 + V(x)|\hat{\psi}_{1,tA,V}|^2\big)=\|\hat{\psi}_{1,tA,V} \|_{\dot\Sigma_{0,V}}+o_t(1).
$$

Thus, without loss of generality, we assume that $\hat{\psi}_{1,tA,V}\rightharpoonup \hat{\psi}_{0}$ weakly in $\Sigma_{0,V}$ as $t\to 0$, for some $\hat{\psi}_{0}\in \Sigma_{0,V}$. 
By Proposition \ref{pr2.5}, we can assume $\hat{\psi}_{1,tA,V}\to \hat{\psi}_{0}$ strongly in $L^p$ for all $2\le p<2^*$. 
Hence, $\|\hat{\psi}_{0}\|^2_2 = 1$ and 
\begin{align*}
\begin{split}
\la_{1,0,V}+o_t(1)
&\ge\hat{\la}_{1,tA,V}= \int_{\R^3}\big(|(i\nabla + tA(x))\hat{\psi}_{1,tA,V}|^2 + V(x)|\hat{\psi}_{1,tA,V}|^2\big)\\
&\ge \int_{\R^3}\big(|\nabla \hat{\psi}_{0}|^2 + V(x)|\hat{\psi}_{0}|^2\big) + o_1(t)\ge \la_{1,0,V} + o_1(t).
\end{split}
\end{align*}
Consequently, $\hat \psi_0$ is an eigenvector of $\hat L_{0,V}$ that corresponds to the simple eigenvalue $\la_{1,0,V}$, so that $\hat\psi_0=e^{i\a}\psi_{1,0,V}$, for a suitable $\a\in \R$.
Moreover,
$$
\lim\limits_{t\to 0}\hat{\la}_{1,tA,V} = \la_{1,0,V}
$$
and
$$
%\begin{equation}\label{teqA.3}
\hat{\psi}_{1,tA,V}\to  e^{i\a}\psi_{1,0,V}\quad  \text{strongly in}\ \Sigma_{0,V}.
$$
%\end{equation}
So, the claim is proved, and the proof is completed.
\end{proof}

\begin{proposition}\label{P2} 
Assume $\mathcal{A}_{\hat{\a}}(x)= {\hat{\a}}\big(\frac{-x_2}{|x|^2},\frac{x_1}{|x|^2},0\big)$, ${\hat{\a}}\in\R$,  and $V_{\hat{\a}}(x) := V(x) + |\mathcal{A}_{\hat\a}(x)|^2$ satisfies \eq{V1} with $A  \equiv 0$. 
If $|{\hat{\a}}|\le 1/2$, then $\hat{\psi}_{1,\mathcal{A}_\a,V}$ is positive and unique up to a constant phase, with
$$
%\label{1224}
\hat{\psi}_{1,\mathcal{A}_{\hat{\a}},V} = e^{i\theta} \psi_{1,0,V_{\hat{\a}}},\qquad \theta\in\R,
$$
where $ \psi_{1,0,V_{\hat{\a}}}$ is defined in Lemma \ref{Ale2.7}. 
Moreover, assume that $V_{\hat{\a}}$ is cylindrically symmetric and radially nondecreasing w.r.t. $(x_1,x_2)$, then $\hat{\psi}_{1,\mathcal{A}_{\hat{\a}},V}$ is cylindrically symmetric and radially decreasing w.r.t. $(x_1,x_2)$. 

\end{proposition}

\begin{proof}  This part is inspired by the Remark of Theorem 7.4 in \cite{Lieb-Loss-AMS-2005}. 
For a function $u$ in ${\Sigma}_{\mathcal{A}_{\hat{\a}},V}$, setting $r^2=x_1^2+x_2^2$ and arguing exactly as for \eq{P29} in Lemma \ref{able5.1},  we obtain
\beq
\label{1223}
\int_{\R^3}\big(|(i\nabla + \mathcal{A}_{\hat{\a}}(x))u|^2 + V(x)|u|^2\big)dx
\ge \int_{\R^3}|\nabla  u|^2dx + \int_{\R^3}\Big(\frac{|\hat{\a}|^2r^2}{(r^2+x^2_3)^2} + V(x)\Big)|u|^2dx,
\eeq
that implies
\begin{multline}
\label{1222}
\inf_{\substack{u\in \Sigma_{\mathcal{A}_{\hat{\a}},V} \\ \|u\|_2 =1}}\left[\int_{\R^3}|(i\nabla+\mathcal{A}_{\hat{\a}}(x))  u|^2 + \int_{\R^3} V(x) \cdot|u|^2\right]
\\
\ge\inf_{\substack{u\in \Sigma_{0,V_{\hat{\a}}} \\ \|u\|_2 =1}}\left[\int_{\R^3}|\nabla  u|^2 + \int_{\R^3}\left(\frac{{\hat{\a}}^2r^2}{(r^2+x^2_3)^2} + V(x)\right)|u|^2\right]
\end{multline}

However, by \cite[Theorem 11.8]{Lieb-Loss-AMS}, the function $\psi_{1,0,V_{\hat{\a}}}$ achieving 
\beq\label{1225}
\inf_{\substack{u\in \Sigma_{0,V} \\ |u|_2 =1}}\left[\int_{\R^3}|\nabla  u|^2 + \int_{\R^3}\left(\frac{{\hat{\a}}^2r^2}{(r^2+x^2_3)^2} + V(r,x_3)\right)|u|^2\right]
\eeq
is unique and positive up to a constant phase. 
Using $\psi_{1,0,V_{\hat{\a}}} $ to test the first infimum in \eq{1222}, it turns out that equality is satisfied in \eq{1222}. 
Consequently, according to \eqref{1223}, the minimizers for the first infimum in \eqref{1222} are also minimizers for \eq{1225}.
So, from the variational characterization of the first eigenvalues, \eq{1223} follows.

Finally, if $V_{\hat{\a}}$ is cylindrically symmetric and  radially nondecreasing w.r.t. $(x_1,x_2)$, then we can proceed as in the proof of part 2 of Theorem \ref{th1.3} and {\cite[Theorem 2]{Bellazzni-et.al.-CMP-2017}}, and conclude the proof.
\end{proof}

%\begin{remark}
%The cylindrical  symmetry  of $v$ can also be obtained by applying the moving plane methods \cite{Gidas-Ni-Nirenberg-CMP-1979}(along the two directions $x_1$ and $x_2$) on the equation
%$$
%-\Delta u + \big(|\a(-x_2/r^2,x_1/r^2,0)|^2 + V(r,x_3)-\la\big)u =0\ \text{in}\ \R^3, u\ge0.
%$$
%By the same argument in Remark \ref{re4.1}, when $V$ is radially symmetric, we can also conclude that $\hat{\psi}_{0,A,V}$ is radially symmetric if $A(x) = \a(-x_2/r^2,x_1/r^2,0)$ with $\alpha\in \R$ satisfying $|\alpha|\le\frac{1}{2}$.
%\end{remark}

\begin{proposition}\label{leA.3}
Let $A_{\rho_k}(x) = \rho_k x^{\bot}$ and $V_{\rho_k}(x) = |x|^{2k} + k \sgn(k-1) - \rho^2_k|x^{\bot}|^2,k\in \N^+$. 
Then there exists a $\rho_{0,k}>0$, such that the first eigenvalue of $(i\nabla + A_{\rho_k})^2 + V_{\rho_k}$ is simple if $|\rho_k|<\rho_{0,k}$.
 Moreover, the first eigenfunction is positive and radial up to a constant phase.
\end{proposition}

\begin{proof}
The proof that the first eigenvalue is simple follows closely that of Proposition \ref{leA.1}, by using the same implicit function argument. 
Here we consider the function
\begin{align*}
\hat{F}(u_1,u_2,\la,t,s)
&= \big[(i\nabla + tx^{\bot})^2 + (|x|^{2k} + k \sgn (k-1) - t^2|x^{\bot}|^2)\big](u+(1+s)\psi_{1,k})\\
&\quad - (\la_{1 ,k}+\la)(u+(1+s)\psi_{1,k}),
\end{align*}
where $\la_{1,k}$ is the first eigenvalue  of the operator $-\Delta + |x|^{2k} + k \sgn(k-1)$ and $\psi_{1,k}$ is the corresponding real first eigenfunction normalized in $L^2$, uniquely determined  by Lemmas \ref{Ale2.7} and \ref{LBar}.

To deal with positivity,  we consider the variational characterization of the first eigenvalue, that $A_{\rho_k}$ is of the form \eq{ueq1.12} and $V_{\rho_k}$ is $G$-invariant. 
Then we can argue as in the proof of part 2 of Theorem \ref{th1.3}.
Namely, we obtain \eq{eq5-6}, so that we can consider a non-magnetic problem, for which the non-negativity of the eigenfunction can be proved in a standard way (see also \eq{abeq4.16}).

Finally, by using standard symmetrization arguments, and the Harnak inequality, we obtain the symmetry and the positivity of the first eigenfunction.

We omit the details.
\end{proof}

Using the previous proposition, we can describe the asymptotic behavior of the solutions provided by Theorem \ref{th1.1}, and of the related Lagrange multipliers.
As a consequence, we will derive the proof of Theorem \ref{th1.6}.

\begin{corollary}\label{crA.4}
Let $A_{\rho}(x) =  \rho x^{\bot}$, $V_{\rho}(x) = |x|^{2}  - \rho^2|x^{\bot}|^2$,  and $(u_{c,A_{\rho},V_\rho}, {\la}_{c,A_{\rho},V_\rho})$ as in Theorem \ref{th1.1}.
Then there exist $\rho_0>0,c_0(\rho_0)>0$ such that
\begin{align*}
\hat{\la}_{1,A_{\rho},V_\rho}= \la_1 = 3,\ \la_{c,A_{\rho},V_\rho}- \la_1 = O(c)\ \ \text{and}\ \ \|u_{c,A_{\rho},V_\rho}-\tilde{l}_{c,A_{\rho},V_\rho}\psi_{1}\|_{\dot{\Sigma}_{A_\rho,V_\rho}}^2/c= o_c(1)
\end{align*}
if $|\rho|<\rho_0$ and $0<c<c_0(\rho_0)$, where $\la_1$ is the first eigenvalue  of $-\Delta + |x|^2$ and $\psi_1(x):=\pi^{-\frac{3}{4}}e^{-\frac{1}{2}|x|^2}$ is the corresponding first eigenfunction, and $\tilde{l}_{c,A_{\rho},V_\rho}= \int_{\R^3}\psi_1u_{c,A_{\rho},V_\rho}$ (see \cite[equation  (IV.1)]{Weinstein-CMP-1983}).

\end{corollary}
\begin{proof}
Observe that Proposition \ref{leA.3}, implies $\hat{\psi}_{1,A_{\rho},V_{\rho}} = e^{i\theta}\psi_1$ for some $\theta\in[0,2\pi)$ when $|\rho|>0$ is small.
Furthermore, by Theorem \ref{th1.3}  we infer that $u_{c,A_{\rho},V_{\rho}}$ is unique, positive, and radially symmetric, up to a constant phase, if $c$ is smaller than a constant depending on $\rho$. 
Now, denote $v_{c,A_\rho,V_{\rho}} = u_{c,A_\rho,V_{\rho}}/\sqrt{c}$. 
Then
$$
-\Delta v_{c,A_\rho,V} + |x|^2v_{c,A_\rho,V} - c|v_{c,A_\rho,V}|^2v_{c,A_\rho,V} = \la_{c,A_\rho,V}v_{c,A_\rho,V}
$$
and
$$
\int_{\R^3}\big(-\Delta v_{c,A_\rho,V} + |x|^2v_{c,A_\rho,V}\big)\psi_1 - c\int_{\R^3}|v_{c,A_\rho,V}|^2v_{c,A_\rho,V}\psi_1 = \la_{c,A_\rho,V}\int_{\R^3}v_{c,A_\rho,V}\psi_1,
$$
which implies
$$
\la_1\tilde{l}_{c,A_{\rho},V}- c\int_{\R^3}|v_{c,A_\rho,V}|^2v_{c,A_\rho,V}\psi_1 = \la_{c,A_\rho,V}\tilde{l}_{c,A_{\rho},V},
$$
where we have used the fact $-\Delta \psi_1 + |x|^2\psi_1=\la_1\psi_1$. 
Then
\begin{equation*}
(\la_1-\la_{c,A_\rho,V})\tilde{l}_{c,A_{\rho},V} = c\int_{\R^3}|v_{c,A_\rho,V}|^2v_{c,A_\rho,V}\psi_1.
\end{equation*}
Arguing as in the proof of part 1 of Theorem \ref{th1.3} we can see that $v_{c,A_\rho,V}\to \psi _1$ in $\Sigma_{A,V}$, and in $L^4$,  as $c\to 0^+$. 
Notice that the argument is independent of $\rho$, when working on the radial, real-valued functions. 
As a consequence,  $|\tilde{l}_{c,A_{\rho},V}|\to 1$ and
\[
\left|\int_{\R^3}|v_{c,A_\rho,V}|^2v_{c,A_\rho,V}\psi_1\right|\to \int_{\R^3}|\psi_1|^4.
\]
As a result, we have
$$
\frac{|\la_1-\la_{c,A_\rho,V}|}{c} =\frac{|\int_{\R^3}|v_{c,A_\rho,V}|^2v_{c,A_\rho,V}\psi_0|}{|\tilde{l}_{c,A_{\rho},V}|}\to \int_{\R^3}|\psi_0|^4,
$$
which implies $\la_1-\la_{c,A_\rho,V} = O(c)$.

\smallskip

Finally, by the analysis above, we have
\begin{align*}
\|v_{c,A_\rho,V}-\tilde{l}_{c,A_{\rho},V}\psi_1\|^2_{\dot{\Sigma}_{A,V}}
 &= \int_{\R^3}|\nabla (v_{c,A_\rho,V} - \tilde{l}_{c,A_{\rho},V}\psi_1)|^2 + |x|^2|v_{c,A_\rho,V}-\tilde{l}_{c,A_{\rho},V}\psi_1|^2\\
 &= \la_{c,A_\rho,V} + c\int_{\R^3}|v_{c,A_\rho,V}|^4  - \la_1|\tilde{l}_{c,A_{\rho},V}|^2\\
 &\to 0
\end{align*}
as $c\to 0^+$. 
% Then,
% \[\|v_{c,A_\rho,V} - \tilde{l}_{c,A_{\rho},V}\psi_1\|^2_2\le C \|v_{c,A_\rho,V}-\tilde{l}_{c,A_{\rho},V}\psi_1\|^2_{\dot{\Sigma}_{A,V}}\to 0 \text{ as } c\to 0^+.\]
This completes the proof.
\end{proof}

\textsf{Proof of Theorem \ref{th1.6}.} Theorem \ref{th1.6} clearly follows from Corollary \ref{crA.4}, taking into account the equivalence of the norm $\|\cdot\|_{\dot\Sigma_{A,V}}$ and  $\|\cdot\|_{\Sigma_{A,V}}$, stated in Proposition \ref{pr2.5}.
\qed

%%%%%%%%%%%%%%%%%%%%%%%%%%%%%%%%%%%%%%%%%%%%%%%%%%%%%%%%%%%%%%%%%%%%%%%%%

\appendix
 \renewcommand{\appendixname}{Appendix~\Alph{section}}
 
 \section{Spectrum analysis for \texorpdfstring{$\hat{L}_{A,V}:= (i\nabla + A)^2+V$}{prova}}
%  \label{AA}

%%%%%%%%%%%%%%%%%%%%%%%%%%%%%%%%%%%%%%%%%%%%%%%%%%%%%%%%%%%%%%%%%%%%%%%%%

In this appendix, we give the spectrum analysis for $\hat{L}_{A,V}$, stated in Lemma \ref{Ale2.8}.

\begin{proof}[\textsf{Proof of Lemma \ref{Ale2.8}}]
For the sake of simplicity, we split the proof into several steps.

\vspace{0.1cm}
\textbf{Step 1}. 
Introduction of the ``inverse operator'' of $\hat{L}^{-1}_{A,V}$, 
\[
\hat{S}_{A,V}: L^2(\R^3,\C)\to L^2(\R^3,\C),
\]
which is bounded, compact, and symmetric.

\vspace{0.1cm}

First observe that, using on $\Sigma_{A,V}$ the inner product
$$
\langle u,v\rangle_{\dot\Sigma_{A,V}}:= \int_{\R^3}(i\nabla + A(x))u\overline{(i\nabla + A(x))v} +  \int_{\R^3}V(x)u\bar{v}
$$
(see Proposition \ref{pr2.5}),
by Riesz representation Theorem for every $g\in L^2$ there exists a unique $u_g \in {\Sigma}_{A,V}$ such that 
$$
\langle u_g,v\rangle_{\dot\Sigma_{A,V}}=
\langle g,v\rangle_{L^2}=\int_{\R^3}g \bar v,\qquad \forall v\in\Sigma_{A,V}.
$$
Moreover, if $\cL_g(\cdot):=\langle g,\cdot\rangle_{L^2}\in L(\Sigma_{A,V})$, then
\beq
\label{dont}
\|u_g\|_{\dot\Sigma_{A,V}}=\|\cL_g\|_{L(\Sigma_{A,V})}.
\eeq
Then, denoting by $\Pi: {\Sigma}_{A,V} \to L^2$ the projection map and setting  
$$
\hat{S}_{A,V}(g):=\Pi (u_g), \qquad g\in L^2,
$$
 Proposition \ref{pr2.5} implies that $\hat{S}_{A,V}:L^2\to L^2$  is a bounded, compact operator, taking into account
$$
\| u_g\|^2_{\dot{\Sigma}_{A,V}}=\langle u_g,u_g\rangle_{ \dot\Sigma_{A,V}} = (g,u_g)\le \|g\|_{L^2}\|u_g\|_{L^2}\le C\|g\|_{L^2}\|u_g\|_{\dot{\Sigma}_{A,V}}.
$$

\smallskip

Regarding symmetry, we note that 
$$
{\langle f,\hat{S}_{A,V} g\rangle}_{L^2}
={ \langle u_f,u_g\rangle}_{\dot\Sigma_{A,V}}
=\overline{ \langle u_g,u_f\rangle}_{\dot\Sigma_{A,V}}
=\overline{\langle g,\hat{S}_{A,V} f\rangle}_{L^2}
=\langle  \hat{S}_{A,V} f,g\rangle_{L^2}\qquad \forall f,g\in L^2.
$$

\vspace{0.2cm}

\textbf{Step 2.} Let $\sigma(\hat{S}_{A,V})$ be the spectrum of $\hat{S}_{A,V}$ and $\sigma_p(\hat{S}_{A,V})$ be the collection of eigenvalues of $\hat{S}_{A,V}$. 
Then there hold

(1) $0\in \sigma(\hat{S}_{A,V})$,

(2) if $\lambda\in \sigma_p(\hat{S}_{A,V})$ then $\lambda\in\R$ and $\lambda>0$,

(3) $\sigma(\hat{S}_{A,V}) - \{0\} = \sigma_p(\hat{S}_{A,V})$,

(4) the eigenvectors of $\hat{S}_{A,V}$ are orthogonal,

(5) if there is a sequence $\{\tilde\lambda_k\}\subseteq \sigma_p(\hat{S}_{A,V})$ with $\tilde\lambda_k\to \tilde\lambda\in[0,\infty]$, then $\tilde\lambda=0$.

\vspace{0.2cm}

{\em Proof of (1)}: If $0\not\in \sigma(\hat{S}_{A,V})$, then $\hat{S}_{A,V}:L^2 \to L^2 $ has a bounded inverse.
So, $I=\hat{S}^{-1}_{A,V}\circ \hat{S}_{A,V}$ is compact, contrary to $\dim L^2 = \infty$. 

\smallskip

{\em Proof of (2)}: if  $\psi\in \Sigma_{A,V}$ verifies $\|\psi\|_2=1$ and $\hat{S}_{A,V} \psi =\tilde \lambda \psi$,
then by \eqref{dont}
$$
\tilde\lambda=\tilde\lambda \langle \psi,\psi\rangle_{L^2}
=\langle \tilde\lambda \psi,\psi\rangle_{L^2}
=\langle \hat{S}_{A,V} \psi,\psi\rangle_{L^2}
=\langle  \psi,\hat{S}_{A,V}\psi\rangle_{L^2}
=  \langle u_\psi,u_\psi\rangle_{\dot\Sigma_{A,V}}>0.
$$

\smallskip

{\em Proof of (3)}: Assume $\tilde{\la} \in\sigma(\hat{S}_{A,V})$, $\tilde{\la} \ne 0$. 
Then if $N(\hat{S}_{A,V}-\tilde{\la} I)=\{0\}$, the Fredholm alternative would imply $R(\hat{S}_{A,V}-\tilde{\la} I) = L^2$. 
But then $\tilde{\la} \not \in \sigma(\hat{S}_{A,V})$, a contradiction.  

\smallskip

{\em Proof of (4)}: let $\tilde \lambda_k,\tilde \lambda_l\in\sigma_p(\hat{S}_{A,V})$, with $\tilde\lambda_k\neq\tilde\lambda_l$, and let $u_k,u_l \in L^2 $ be corresponding non-zero eigenvectors, then 
$$
\tilde \lambda_k\langle u_k,u_l\rangle_{L^2}=
 \langle \hat{S}_{A,V} u_k,u_l\rangle_{L^2}=
 \langle  u_k,\hat{S}_{A,V}u_l\rangle_{L^2}=
 \tilde \lambda_l\langle u_k,u_l\rangle_{L^2}\quad \Longrightarrow\quad  \langle u_k,u_l\rangle_{L^2}=0.
 $$

\smallskip

{\em Proof of (5)}:
assume $\tilde{\psi}_k\in L^2$, with $\|\tilde\psi_k\|_2=1$, satisfies $\hat{S}_{A,V}\tilde{\psi}_k = \tilde{\la}_k\tilde{\psi}_k$, for $k\in\N^+$, and let us denote by $V_k$ the subspace of $L^2$ spanned by $\{\tilde{\psi}_1,\ldots,\tilde{\psi}_k\}$. 
Then $V_{k-1}\subsetneqq V_{k}$ for every $k=2,\ldots$,  and $\tilde \psi_k$ is orthogonal to $V_{k-1}$.
Next,  observe that $(\hat{S}_{A,V}-\tilde{\la}_k I)V_k\subset V_{k-1}$, and that  $V_{l-1}\subsetneqq V_l \subsetneqq V_k$, for  $k>l$. 
Thus
\beq
\label{1758}
\left\|\frac{\hat{S}_{A,V}\tilde{\psi}_k}{\tilde{\la}_k} - \frac{\hat{S}_{A,V}\tilde{\psi}_l}{\tilde{\la}_l}\right\|_2 
= \left\|\frac{\hat{S}_{A,V}\tilde{\psi}_k - \tilde{\la}_k \tilde{\psi}_k}{\tilde{\la}_k} - \frac{\hat{S}_{A,V}\tilde{\psi}_l- \tilde{\la}_l \tilde{\psi}_l}{\tilde{\la}_l} + \tilde{\psi}_k-\tilde{\psi}_l\right\|_2
\ge 1,
\eeq
for $\hat{S}_{A,V}\tilde{\psi}_k - \tilde{\la}_k \tilde{\psi}_k,\hat{S}_{A,V}\tilde{\psi}_l- \tilde{\la}_l \tilde{\psi}_l,\tilde{\psi}_l\in V_{k-1}$. 
Then if $\tilde{\la}_k\to \tilde{\la}\ne 0$, we have that $\{\tilde{\psi}_k/\tilde{\la}_k\}$ is a bounded sequence. 
However, \eqref{1758} implies that  $\Big\{\frac{\hat{S}_{A,V}\tilde{\psi}_k}{\tilde{\la}_k}\Big\}$ has no convergent subsequence in $L^2$, contradicting the compactness of $\hat{S}_{A,V}$. 
Hence $\tilde{\la}=0$ and (5) is proved.

\vspace{0.1cm}

\textbf{Step 3.} There exists a countable basis of $L^2(\R^3,\C)$ consisting of eigenvectors of $\hat{S}_{A,V}$.

\vspace{0.1cm}

Let $\{\tilde{\la}_k\}_{i\in I}$ a family of distinct eigenvalues of $\hat{S}_{A,V}$, where $I=\{1,\ldots,n\}$ or $I=\N^+$.
Then, we write $V_k = N(\hat{S}_{A,V} - \tilde{\la}_kI)$, $k\in I$, and observe that, by the Fredholm alternative, it holds
$$
0<\dim V_k<\infty.
$$
Let $\widetilde{V}$ be the smallest subspace of $L^2$ containing $V_k$, $ \forall k\in I$:
\[
\widetilde{V} = \left\{\sum_{k=1}^ma_ku_k\ :\ m\in I,\, u_k\in V_k,\, a_k\in\C\right\}.
\] 
We claim that $\widetilde{V}$ is dense in $L^2$. 
Notice that the claim also implies $I=\N$.
Clearly, $\hat{S}_{A,V}(\widetilde{V})\subset \widetilde{V}$. 
Furthermore, $\hat{S}_{A,V}(\widetilde{V}^{\bot})\subset \widetilde{V}^{\bot}$ because $(\hat{S}_{A,V}u^{\bot},v) = {(u^{\bot},\hat{S}_{A,V}v)}=0$ if $u^{\bot}\in\widetilde{V}^{\bot}$ and $v\in \widetilde{V}$.

Denote $\widetilde{S}_{A,V} := \hat{S}_{A,V}|_{\widetilde{V}^{\bot}}$. 
It can be easily checked that $\widetilde{S}_{A,V}$ is compact and symmetric. 
In addition $\sigma(\widetilde{S}_{A,V}) = \{0\}$, because any eigenvalue of $\widetilde{S}_{A,V}$ would also be an eigenvalue of $\hat{S}_{A,V}$. 
Based on this, we obtain
\begin{equation}
\label{Adeq2.7} 
(\widetilde{S}_{A,V}u^{\bot},u^{\bot})=0\ \text{ for any}\ u^{\bot}\in \widetilde{V}^{\bot}. 
\end{equation}
Indeed,  setting 
$$
m :=\inf_{v\in\widetilde{V}^{\bot},\, \|v\|_2=1}( \tilde{S}_{A,V}v,v)\ \text{ and }\ M :=\sup_{v\in\widetilde{V}^{\bot},\, \|v\|_2=1}(\tilde{S}_{A,V}v,v),
$$
it is well known that $m,M\in \sigma(\widetilde{S}_{A,V})$.
Therefore, if \eq{Adeq2.7} does not hold, then either $0 \neq M \in \sigma(\widetilde{S}_{A,V}) = \{0\}$ or $0 \neq m \in \sigma(\widetilde{S}_{A,V}) = \{0\}$, both of which are contradictions.

Now, by the claim \eq{Adeq2.7},  letting $u^{\bot},v^{\bot}\in\widetilde{V}^{\bot}$, and $\theta\in\R$, it holds 
$$
2\Re(\widetilde{S}_{A,V}u^{\bot},e^{i\theta}v^{\bot})=\big(\widetilde{S}_{A,V}(u^{\bot}+e^{i\theta}v^{\bot}),u^{\bot}+e^{i\theta}v^{\bot}\big) - (\widetilde{S}_{A,V}u^{\bot},u^{\bot})-(\widetilde{S}_{A,V}e^{i\theta}v^{\bot},e^{i\theta}v^{\bot}) = 0.
$$
Hence $\widetilde{S}_{A,V} = 0$.
Consequently $\widetilde{V}^{\bot}\subset N(\hat{S}_{A,V})$, where $N(\hat{S}_{A,V})= \{0\}$ by (2) in Step 2. 
Thus, $\widetilde{V}$ is dense in $L^2$.

Finally, collecting the orthonormal basis for each subspace $V_k$, $k\in\N$, we obtain an orthonormal basis of eigenfunctions for $L^2$.

\vspace{0.1cm}

\textbf{Step 4.} {Conclusion.}

\vspace{0.1cm}

By Steps 1-3, we know that the eigenvalues of $\hat{S}_{A,V}$ are a positive sequence $\{\breve{\la}_{k,A,V}\}$ that satisfies $\lim\limits_{k\to\infty}\breve{\la}_{k,A,V} = 0$. 
Moreover, the corresponding eigenfunctions $\{\breve{\psi}_{k,A,V}\}$ are an orthonormal basis of $L^2$. 
Returning to $\hat{L}_{A,V}$, we have
$$
\hat{L}_{A,V}\breve{\psi}_{k,A,V} = \hat{L}_{A,V}\left(\frac{1}{\breve{\la}_{k,A,V}}\hat{S}_{A,V}\breve{\psi}_{k,A,V}\right) = \frac{1}{\breve{\la}_{k,A,V}}\breve{\psi}_{k,A,V},
$$
that is, all values  $\hat{\la}_{k,A,V} := \frac{1}{\breve{\la}_{k,A,V}}$, $k\in\N^+$,  are  eigenvalues of $\hat{L}_{A,V}$. 
Moreover, $\{\breve{\psi}_{k,A,V}\}_{k\in\N^+}$ are the eigenfunctions corresponding to $\{\hat{\la}_{k,A,V}\}_{k\in\N^+}$. 
This gives (i)-(iii).
\end{proof}

\vspace{0.5cm}

%%%%%%%%%%%%%%%%%%%%%%%%%%%%%%%%%%%%%%%%%%%%%%%%%%%%%%%%%%%%%%%%%%%%%%%%%%%%%%%%%%%%%%%%%%%%%%%%

\section{Explicit expressions of the threshold on mass}\label{AB}
\renewcommand{\appendixname}{Appendix~\Alph{section}}

%%%%%%%%%%%%%%%%%%%%%%%%%%%%%%%%%%%%%%%%%%%%%%%%%%%%%%%%%%%%%%%%%%%%%%%%%%%%%%%%%%%%%%%%%%%%%%%%

In this appendix, we give explicit values of $c_{A,V}(r)$ and $\tilde{c}_{A,V}(r)$, from Theorem \ref{th1.1}.
Moreover, we evaluate the values $c_*$ and $\tilde c_*$ arising in Theorem \ref{th1.1Bis}.
Finally, in Remarks \ref{reB.1} and \ref{reB.2} we provide estimates for the concrete cases in Examples \ref{ex1} and \ref{ex2}, respectively.

\smallskip

 For fixed $r>0$,  a straight computation shows that there exists $\sigma_r\in (0,1)$ such that
 \beq
 \label{1549}
 c_{A,V}(r)=F_{A,V}(\si_r,r) = \frac{2}{\hat{\la}_{1,A,V}} \, \frac{r}{2 + \ga r^2 + \sqrt{4\ga r^2 + \ga^2r^4}},
 \eeq
(see \eq{1820}),
%  \begin{align*}\label{eeq1.5}
%  c_{A,V}(r)% &:= \max_{\si\in(0,1)}F_{A,V}(\si,r) = \frac{r}{\hat{\la}_{0,A,V}}\frac{2 + \ga r^2 - \sqrt{4\ga r^2 + \be^2r^4}}{2}\\
% &
%  := \max_{\si\in(0,1)}F_{A,V}(\si,r) =F_{A,V}(\si_r,r) = \frac{2}{\hat{\la}_{1,A,V}} \, \frac{r}{2 + \ga r^2 + \sqrt{4\ga r^2 + \ga^2r^4}},
%  \end{align*}
where
\[\ga := \frac{C^8_{3,4}}{4\hat{\la}_{1,A,V}}=\frac{1}{\hat{\la}_{1,A,V}\|W_4\|^4_2}\]
 and $W_4$ is the ground state of
 \[-\Delta W_4 + \frac{1}{3}W_4=\frac{2}{3} W_4^3  \text{  in   }  \R^3.\]
The function $c_{A,V}(r)$ attains its absolute maximum value  for 
$$
r_*=\frac{2}{\sqrt{3\ga}}= \frac 23\sqrt{3 \hat\lambda_{1,A,V}}\, \cdot  \|W_4\|_2^2,
$$
so that we have the existence of our local minimum solution for problem \eqref{eq1.1} for every mass $0<c< c_*$, with the following lower estimate for the threshold  
 \begin{equation}
 \label{1536}
c_*: = c_{A,V}(r_*)  = \frac{4\sqrt{3}}{9C^4_{3,4}\sqrt{\hat{\la}_{1,A,V}}}=\frac{2\sqrt{3}\|W_4\|^2_2}{9\sqrt{\hat{\la}_{1,A,V}}}.
 \end{equation}

Now, to evaluate $\tilde c_*$, let us set 
$m:= \frac{\min\{1 - 3\va_0,3\alpha_0\min\{1,\bar \a\}\}}{6\hat{\la}_{1,A,V}}$.
Then from \eq{1517} and elementary calculus we have
\beq
\label{1537}
\tilde c_*:=
\  \left\{\begin{array}{ccl} 
c_* & &\mbox{ if } mr_*\ge c_* \\
m \hat r = c_{A,V}(\hat r) & \mbox{with }\hat r> r_* &\mbox{ if } mr_*<  c_*.
\end{array}
\right. 
\eeq

\begin{remark}\label{reB.1}
Here, we refer to Example \ref{ex1}: let $A(x) = A_\rho(x) = \rho x^\bot$, $V(x) = V_\rho(x) = |x|^2 - \rho^2 |x^\bot|^2$. 
Then:

\smallskip

1. Let $|\rho|<\rho_0$, with $\rho_0$ being the constant given in Corollary \ref{crA.4}. 
According to Corollary \ref{crA.4}, it holds $\hat{\la}_{1,A,V} = 3$. 
In this case, our estimate \eq{1536} of the range of $c$ for the {\it existence of local minimizer}  is
\begin{equation*}
%\label{eqB.3}
0<c<c ^* = \frac{2}{9}\|W_4\|^2_2.
\end{equation*}

2. For the threshold of {\it existence of an energy ground state}, we fix in \eqref{1517} the following parameters that, in particular, guarantee that \eqref{V2} is true.
Let $r = r_*$,  $|\rho| < \min\{\frac{2}{3},\rho_0\}$, $\va_0 = \frac{3}{10}$ and $\bar \a = 1- \rho^2$. For $\alpha_0$, recalling the computation of $\tilde{T}_{A,V,\va_0}$ in Example \ref{ex1}, we have
$$
\inf_{|x|>0}\frac{\tilde{T}_{A,V,\va_0}}{V(x)} 
= \inf_{|x|>0}\frac{5 - 10 \rho^2 + \frac{10\rho^2x^2_3}{|x|^2}}{6-6\rho^2 + \frac{6\rho^2x^2_3}{|x|^2}} 
= \min_{t\in[0,1]}\frac{5 - 10 \rho^2 + 10\rho^2t}{6-6\rho^2 + 6\rho^2t} 
= \frac{5-10\rho ^2}{6-6\rho^2},
$$
so we can choose  $\a_0 = \frac{5-10\rho ^2}{6-6\rho^2}$ in \eq{V2}.
Then, by \eqref{1517}, the local minimum in point 1 of Theorem \ref{th1.1} is a ground state if  
$$
0<c<\tilde c_{A,V}(r^*)=\frac{\|W_4\|^2_2}{90}.
$$
Clearly, we have also obtained an estimate for $\tilde c_*$ in point 2 of Theorem \ref{th1.1Bis} (see \eq{1354}):   
$$
\tilde c_*\ge \frac{\|W_4\|^2_2}{90}.
$$
\end{remark}

\begin{remark}\label{reB.2}
Here, we refer to Example \ref{ex2}, with  $|{\hat{\a}}|\le 1/2$. 
The assumption $|{\hat{\a}}|\le 1/2$ ensures that the problem \eq{eq1.1} and the minimization problem \eq{abeq1.2} are reduced to \eq{abeq1.5} and \eq{abeq4.11}, respectively, i.e., $A(x)\equiv 0$ and the electric potential becomes $V(x) + |\mathcal{A}_{\hat{\a}}(x)|^2$.
In the following, we consider the two special cases $V(x)+ |\mathcal{A}_{\hat{\a}}(x)|^2 = |x|^2$ and $ V(x) +  |\mathcal{A}_{\hat{\a}}(x)|^2 = |x^\bot|^2$,   considered in \cite{Bellazzni-et.al.-CMP-2017}, and provide some estimates of the mass thresholds in Theorems \ref{th1.1} and \ref{th1.1Bis},  
taking advantage of \cite[equation (IV.l)]{Weinstein-CMP-1983}. 
Observe that in the second case,  the potential $x^2_1+x^2_2$ in $\R^3$ does not meet the growth condition in \eq{V1}. 
But a symmetrization argument, as in \cite[Page 241]{Bellazzni-et.al.-CMP-2017}, implies that a minimizing sequence for $m^r_c$ is also compact. 
Hence, the conclusions in Theorem \ref{th1.1} also hold in this case. 

\smallskip

1. In the first case, $\hat{\la}_{1,0,|x|^2} = 3$.
Since $V(x)+ |\mathcal{A}_{\hat{\a}}(x)|^2 = |x|^2$ satisfies \eq{V1}-\eq{V2}, all the conclusions of Theorems \ref{th1.1}, \ref{th1.1Bis} and \ref{th1.3} also hold for \eq{abeq1.5}. 
Hence \eq{1549} states that for any $r>0$ the local minimum $m^r_c$ can be achieved and there exists a solution $(u_{c,0,|x|^2},\la_{c,0,|x|^2})$ of \eq{abeq1.5} if
$$
0<c<c_{0,|x|^2}(r) = \frac{2\|W_4\|^4_2r}{6\|W_4\|^4_2 + r^2 + \sqrt{12\|W_4\|^4_2 r^2 + r^4}}.
$$
Next, for any nontrivial solution $v$ of \eq{abeq1.5} with  $J(v)\le m^r_c$, by the Pohozaev identity in Proposition \ref{pr3.2}  and \eq{eq5-1}, we find
\beq\label{1624}
 \frac{3 }{2}\, c=\frac{\hat{\la}_{1,0,|x|^2}}{2}\, c\ge m_c^r\ge J(v) = \frac{1}{6}\int_{\R^3}|\nabla v|^2 + \frac{5}{6}\int_{\R^3}|x|^2 |v|^2\ge\frac{1}{6}\|v\|^2_{\dot{\Sigma}_{0,|x|^2}}.
\eeq

Then $v\in B(r)$ if $c\le r/9$, and so  $u_{c,0,|x|^2}$ is an energy ground state if $0<c<\min\{r/9,$ $c_{0,|x|^2}(r)\}$. 
In particular, letting $r = r^* = 2\|W_4\|^2_2$, we have
$$
c_*=c_{0,|x|^2}(r^*) = \frac{2}{9}\|W_4\|^2_2
$$
and an accurate range of the existence of the energy ground state (local minimizer becoming an energy ground state) is
$$
0<c<\frac{2}{9}\|W_4\|^2_2.
$$
Notice that in this case we have verified $c_*=\tilde c_*$.

\smallskip

2. In the second case, $\hat{\la}_{1,0,|x^\bot|^2} = 2$ (\cite[Lemma 2.1]{Bellazzni-et.al.-CMP-2017}). 
Then
$$
c_{0,|x^\bot|^2}(r) = \frac{2\|W_4\|^4_2r}{4\|W_4\|^4_2 + r^2 + \sqrt{8\|W_4\|^4_2 r^2 + r^4}}.
$$
Letting $r = r^* = \frac{2\sqrt{6}}{3}\|W_4\|^2_2$, we obtain
$$
c_*=c_{0,|x^\bot|^2}(r^*) = \frac{\sqrt{6}}{9}\|W_4\|^2_2.
$$
Now, taking into account  \eq{1624} and arguing as before, we conclude that a local minimum solution in $S_c\cap B(r^*)$ is also a ground state solution when $c\le r^*/6=\frac{\sqrt{6}}{9}\|W_4\|^2_2$.
Then, since $\min\{r^*/6,c_{0,|x^\bot|^2}(r^*)\} = \frac{\sqrt{6}}{9}\|W_4\|^2_2$, an accurate range of the existence of energy ground state is
$$
0<c<\tilde c_*=c_*=\frac{\sqrt{6}}{9}\|W_4\|^2_2.
$$
 Also in  this case we have verified $c_*=\tilde c_*$.
\end{remark}

\vspace{.5cm}

\noindent {\bf Acknowledgements} Xiaoming An is supported by  Guizhou Provincial Basic Research Program (Natural Science) under grant Qiankehejichu-zk[2025] General 233.
Lun Guo is supported by Hubei Provincial Natural Science Foundation of China (No.2024AFB839).
Xiao Luo is supported by National Natural Science Foundation of China (No.~12471103) and Anhui Provincial Natural Science Foundation (No.2308085MA05).
Riccardo Molle is supported by the INdAM-GNAMPA group
by the MIUR Excellence Department Projects awarded to the Department of Mathematics, University of Rome ``Tor Vergata'', CUP E83C18000100006 and CUP E83C23000330006, and by
the MUR-PRIN-2022AKNSE4 003 ``Variational and Analytical aspects of Geometric PDEs''.

\vspace{.5cm}

 \noindent {\bf Data Availability} Data sharing not applicable to this article as no datasets were generated or analysed during the current study.

\vspace{.5cm}

\noindent\textbf{Conflict of interest} The authors declare there is no conflicts of interest.

\end{document}